\documentclass[12pt]{article}
\usepackage{appendix}
\usepackage{amsthm}
\usepackage{amsmath,bm}
\usepackage{enumerate}
\usepackage{amssymb}
\usepackage{latexsym}
\usepackage{amsfonts}
\usepackage{color}
\usepackage{secdot}
\usepackage{mathrsfs}
\usepackage{subfigure}
\usepackage{epsfig}
\usepackage{natbib}
\newtheorem{assumption}{Assumption}[section]
\newtheorem{theorem}{Theorem}[section]
\newtheorem{counterexample}{Counterexample}[section]
\newtheorem{example}{Example}[section]
\newtheorem{corollary}{Corollary}[section]
\newtheorem{definition}{Definition}[section]
\newtheorem{lemma}{Lemma}[section]

\usepackage{hyperref}
\makeatletter \oddsidemargin  0.95in \evensidemargin 0.95in
\usepackage[affil-it]{authblk}
\usepackage{etoolbox}
\usepackage{lmodern}

\makeatletter
\patchcmd{\@maketitle}{\LARGE \@title}{\fontsize{20}{19.2}\selectfont\@title}{}{}
\makeatother

\begin{document}
	\title{Stochastic orderings between two finite mixtures with inverted-Kumaraswamy distributed components}
	\author{{\large{Raju {\bf Bhakta}$^{1}$\thanks {Email address: bhakta.r93@gmail.com,~raju\_bhakta@nitrkl.ac.in,~raju\_bhakta.maths@yahoo.com},~Pradip {\bf Kundu}$^{2}$\thanks {Email address: kundu.maths@gmail.com},~Suchandan {\bf Kayal}$^{3}$\thanks {Email address (corresponding author): kayals@nitrkl.ac.in,~suchandan.kayal@gmail.com}~and~Morad {\bf Alizadeh}$^{4}$\thanks {Email address: m.alizadeh@pgu.ac.ir,~moradalizadeh78@gmail.com}}}}
	\maketitle
	\noindent{\it$^{1,3}$Department of Mathematics, National Institute of Technology Rourkela, Rourkela-769008, Odisha, India.}\\
	{\it$^{2}$School of Computer Science and Engineering, XIM University, Bhubaneswar, Odisha, India.}\\
	{\it$^{4}$Department of Statistics, Faculty of Intelligent Systems Engineering and Data Science, Persian Gulf University, Bushehr, 75169, Iran.}
	
	\begin{center}
		\noindent{\bf Abstract}
	\end{center}
	In this paper, we consider two finite mixture models (FMMs), with  inverted-Kumaraswamy distributed components' lifetimes. Several stochastic ordering results between the FMMs have been obtained. Mainly, we focus on three different cases in terms of the heterogeneity of parameters. The usual stochastic order between the FMMs have been established when heterogeneity presents in one parameter as well as two parameters. In addition, we have also studied ageing faster order in terms of the reversed hazard rate between two FMMs when heterogeneity is in two parameters. For the case of heterogeneity in three parameters, we obtain the comparison results based on reversed hazard rate and likelihood ratio orders. The theoretical developments have been illustrated using several examples and counterexamples.
	
	\section{Introduction}\label{section1}
	The FMMs have been widely used in many study
	areas, including biology, reliability, and survival analysis. As a
	result, both theorists and practitioners have shown a great deal of
	interest in these models. Due to its unique ability to model
	heterogeneous data, whose pattern cannot be produced by a single
	parametric distribution, the mixture model (MM) has acquired a lot
	of appeal. An unknown model is developed by mixing a collection of
	homogeneous subpopulations (infinite) in order to capture this
	heterogeneity. Keep in mind that the mixing is carried out over a
	latent parameter, which is regarded as a random variable (RV) and
	chosen from an unknown mixing distribution. Here, we refer to it as
	mixing proportions or weights throughout this paper. There are many
	circumstances in which FMMs spontaneously occur. For more
	information on the different applications of FMMs, see
	\cite{lindsay1995mixture}, \cite{McLachlan2000},
	\cite{amini2017stochastic}, and \cite{wu2001}. For instance,
	\begin{itemize}
		\item A FMM can be used in reliability theory to model the ``failure time'' of a system. The model assumes that the ``failure time'' is a mixture of two or more distributions, as usually there is more than one reason causing the failures of a
		component or system (\citealp{amini2017stochastic}).
		\item To study the distribution of time to death after a major cardiovascular surgery, FMM is useful. Here, one may consider that the lifetime of such patients after surgery contains three phases of time. In the first phase, that is, immediately after surgery, the death risk is relatively high. In the next phase, the hazard rate remains constant upto some certain time. Then, in the final phase, the risk of death of the patient increases. The convenient way to model this situation is to adopt a MM with three components. In each phase, separate parametric model can be assigned to each (here three) components (\citealp{McLachlan2000}). 
		
		\item A FMM can be used in biological sciences to model the distribution of gene expression levels across different cell types. The model assumes that the gene expression levels are a mixture of two or more distributions, such as a normal distribution and a gamma distribution (\citealp{schork1996mixture} and \citealp{McLachlan2000}).
		
	\end{itemize}
	Also, FMMs are used in a variety of real-life applications, such as clustering, image segmentation, anomaly detection, and speech recognition. They are also used in medical diagnosis, market segmentation, and customer segmentation, etc. In this paper, we have considered two FMMs for inverted-Kumaraswamy ($\mathcal{IK}$) distributed components.

	The topic of stochastic comparisons between two FMMs has been extensively studied. For more details, see \cite{shaked2007stochastic}, \cite{navarro2008likelihood}, \cite{navarro2016stochastic}, \cite{amini2017stochastic}, and so on. \cite{hazra2018stochastic} studied stochastic comparisons of finite mixtures (FMs) where the subpopulations are from semiparametric models, that is, the scale model, proportional hazard rate model, and proportional reversed hazard rate model. \cite{barmalzan2021stochastic} focused on two finite $\alpha$-MMs and established sufficient conditions for comparing two $\alpha$-MRVs. Please see, \cite{asadi2018mixture} for some properties of the $\alpha$-MM. \cite{sattari2021stochastic} investigated MMs with generalized Lehmann distributed components and presented several ordering outcomes. \cite{barmalzan2022orderings} studied two FMMs with location-scale family distributed components and established some stochastic comparison results between them in terms of the usual stochastic order and the reversed hazard rate order. By utilizing the majorization idea, \cite{nadeb2022new} have been studied a stochastic comparison for two FMs in terms of usual stochastic order, hazard rate order, and reversed hazard rate order. \cite{panja2022stochastic} established ordering results between two finite mixture random variables (FMRVs), where the mixing components are based on proportional odds, proportional hazards, and proportional reversed hazards models. \cite{kayal2023some} obtained some ordering results between two FMMs considering general parametric families of distributions. Mainly, the authors established sufficient conditions for usual stochastic order based on $p$-larger order and reciprocally majorization order.Very recently, \cite{bhakta2023stochastic} considered similar general parametric families of distributions as in \cite{kayal2023some}, and then examined various ordering results with respect to usual stochastic order, hazard rate order, and reversed hazard rate order between two FMMs.
	
	Inverted distributions have several applications in various fields,
	including econometrics, life testing, biology, engineering sciences,
	and medicine. Additionally, it is used in reliability theory,
	survival analysis, financial literature, and environmental research.
	For more details on inverted distributions and its applications, see
	\cite{abd2013general}. \cite{abd2017inverted} developed the
	$\mathcal{IK}$ distribution by using the transformation $x=t^{-1}-1$
	from the Kumaraswamy ($\mathcal{K}$) distribution, that is,
	$T\thicksim\mathcal{K}(\alpha,\beta)$, where $\alpha$ and $\beta$
	are the shape parameters. Then $X$ has a $\mathcal{IK}$ distribution
	with cdf and pdf as
	\begin{eqnarray}\label{eq1.2}
		F_X(x)\equiv F(x;\alpha,\beta)=(1-(1+x)^{-\alpha})^{\beta},~x>0,~\alpha>0,~\beta>0
	\end{eqnarray}
	and
	\begin{eqnarray}\label{eq1.3}
		f_X(x)\equiv f(x;\alpha,\beta)=\alpha \beta
		(1+x)^{-\alpha-1}(1-(1+x)^{-\alpha})^{\beta-1},~x>0,~\alpha>0,~\beta>0,
	\end{eqnarray}
	respectively, where $\alpha$ and $\beta$ both are shape parameters. We note that the curves of the pdf and hazard function show that the $\mathcal{IK}$ distribution exhibits a long right tail, compared with other commonly used distributions. As a result, it affects long term reliability predictions, producing optimistic predictions of rare events occurring in the right tail of the distribution compared with other well-known distributions.
	Here, we use the notation $X\thicksim\mathcal{IK}(\alpha,\beta)$ for
	convenience. Many well-known distributions fall under the
	$\mathcal{IK}$ distribution as special cases, e.g., Lomax
	distribution (for $\beta=1$), Beta Type II distribution (for
	$\alpha=1$) and the log-logistic distribution (for $\alpha=\beta=1$)
	(\citealp{abd2017inverted}). Also, using appropriate
	transformations, the $\mathcal{IK}$ distribution can be transformed
	to many well-known distributions such as exponentiated exponential
	and Weibull, generalized uniform, generalized Lomax, beta Type II
	and F-distribution, Burr Type III and log logistic distributions
	(\citealp{abd2017inverted}). 
	\cite{abd2017inverted} showed with many real date sets how well the
	$\mathcal{IK}$ distribution fits those real data.
	
	This paper focuses on the stochastic comparison results between two finite mixture models follow $\mathcal{IK}$ distributed components. The goal of this paper is to obtain sufficient conditions, for which two finite mixture random variables with $\mathcal{IK}$ distributed component lifetimes are comparable in the sense of the usual stochastic order, reversed hazard rate order, likelihood ratio order, and ageing faster order in terms of reversed hazard rate order.
	
	The main contributions and organization of the paper are as follows. In the next section, we present several definitions and lemmas which are essential to obtain our main results. Section \ref{section3} contains three subsections, with a description on the proposed model. In Subsection \ref{subsection3.1}, we establish usual stochastic order between two FMMs based on the concepts of the weak supermajorization and weak submajorization orders. Subsection \ref{subsection3.2} deals with the ordering results when there is heterogeneity in two parameters. Here, we obtain comparison results with respect to the usual stochastic order and ageing faster order in terms of the reversed hazard rate. In Subsection \ref{subsection3.3}, we examine ordering results between the FMMs with respect to the reversed hazard rate and likelihood ratio orders. Besides the theoretical contributions, we present many examples and counterexamples for the validation and justification.

	Throughout the paper, the term ``increasing" refers
	``nondecreasing", while the term ``decreasing" will mean
	``nonincreasing". For any differentiable function $\eta(\cdot)$, we
	write $\eta^{\prime}(t)$ to represent the first order derivative of
	$\eta(t)$ with respect to $t$. The partial derivative of
	$\xi(\boldsymbol{x})$ with respect to its $k$th component is denoted
	as ``$\xi_{(k)}(\boldsymbol{x})=\partial\xi(\boldsymbol{x})/\partial
	x_k$", for $k=1,\ldots,n$. Also, ``$\stackrel{sign}{=}$" is used to
	indicate that the signs on both sides of an equality are the same.
	We use the notation $\mathbb{R}=(-\infty,+\infty)$,
	$\mathbb{R}^+=[0,+\infty)$, $\mathbb{R}_n=(-\infty,+\infty)^n$, and
	$\mathbb{R}_n^+=[0,+\infty)^n$, respectively.
	
	\section{The basic definitions and some prerequisites}\label{section2}
	This section presents some preliminary definitions and results, which are essential to establish our main results in subsequent section. Let $U$ and $V$ be two continuous and non-negative independent random variables with pdfs $f_U(\cdot)$ and $g_V(\cdot)$, cdfs $F_U(\cdot)$ and $G_V(\cdot)$, sfs $\bar{F}_U(\cdot)=1-F_U(\cdot)$ and $\bar{G}_V(\cdot)=1-G_V(\cdot)$, and rhs $\tilde{r}_U(\cdot)=f_U(\cdot)/F_U(\cdot)$ and $\tilde{r}_V(\cdot)=g_V(\cdot)/G_V(\cdot)$, respectively.
	\begin{definition}\label{definition2.1}
		The random variable $U$ is smaller than $V$ in the sense of
		\begin{itemize}
			\item[(i)] usual stochastic order (abbreviated as $U\leq_{st}V$), if $\bar{F}_U(x)\leq\bar{G}_V(x)$ for all $x\in\mathbb{R}^+$;
			
			\item[(ii)] reversed hazard rate order (abbreviated as $U\leq_{rh}V$), if $G_V (x)/F_U(x)$ is increasing in $x$ for all $x\in\mathbb{R}^+$; or equivalently, if $\tilde{r}_U(x)\leq\tilde{r}_V(x)$ for all $x\in\mathbb{R}^+$;
			
			\item[(iii)] likelihood ratio order (abbreviated as $U\leq_{lr}V$), if $g_V (x)/f_U(x)$ is increasing in $x$ for all $x\in\mathbb{R}^+$;
			
			\item[(iv)] ageing faster order with respect to reversed hazard rate order (abbreviated as $U\leq_{R-rh}V$), if $\tilde{r}_U(x)/\tilde{r}_V(x)$ is decreasing in $x$ for all $x\in\mathbb{R}^+$.
		\end{itemize}
	\end{definition}
	It is important to note that $U\leq_{lr}V\Rightarrow U\leq_{rh}V\Rightarrow U\leq_{st}V$. \cite{shaked2007stochastic} provide a comprehensive overview of stochastic orders and their applications. The idea of majorization and related orders are then discussed, which are highly helpful in establishing the main results in the subsequent sections. Let $\mathbb{R}_n$ be an $n$-dimensional Euclidean space. Let $\boldsymbol{\varsigma}=(\varsigma_1,\ldots,\varsigma_n)$ and $\boldsymbol{\varepsilon}=(\varepsilon_1,\ldots,\varepsilon_n)$ be two real vectors. Further, let $\varsigma_{(1)}\leq\ldots\leq \varsigma_{(n)}$ and $\varepsilon_{(1)}\leq\ldots\leq \varepsilon_{(n)}$ denote the increasing arrangements of the components of the vectors $\boldsymbol{\varsigma}$ and $\boldsymbol{\varepsilon}$, respectively.
	\begin{definition}\label{definition2.2}(\citealp{marshall2011inequalities}) The vector $\boldsymbol{\varsigma}$ is said to be
		\begin{itemize}
			\item majorize the vector $\boldsymbol{\varepsilon}$ (abbreviated as $\boldsymbol{\varsigma}\stackrel{m}{\succcurlyeq}\boldsymbol{\varepsilon}$) if $\sum_{j=1}^{i}\varsigma_{(j)}\leq\sum_{j=1}^{i}\varepsilon_{(j)}$ for $i=1,\ldots,n-1$, and $\sum_{j=1}^{n}\varsigma_{(j)}=\sum_{j=1}^{n}\varepsilon_{(j)}$;
			
			\item weakly supermajorize the vector $\boldsymbol{\varepsilon}$ (abbreviated as $\boldsymbol{\varsigma}\stackrel{w}{\succcurlyeq}\boldsymbol{\varepsilon}$), if $\sum_{j=1}^{i}\varsigma_{(j)}\leq\sum_{j=1}^{i}\varepsilon_{(j)}$, for $i=1,\ldots,n$;
			
			\item weakly submajorize the vector $\boldsymbol{\varepsilon}$ (abbreviated as $\boldsymbol{\varsigma}\succcurlyeq_w\boldsymbol{\varepsilon}$), if $\sum_{j=i}^{n}\varsigma_{(j)}\geq\sum_{j=i}^{n}\varepsilon_{(j)}$, for $i=1,\ldots,n$.
		\end{itemize}
	\end{definition}
	
	It is obvious that the majorization order implies both weak supermajorization and weak submajorization orders. In the following, we will present a definition that demonstrates how the Schur-convex function is able to maintain the ordering of the majorization.
	
	\begin{definition}\label{definition2.3}(\citealp{marshall2011inequalities}) A real valued function $\Xi$ defined on a set $\mathbb{S}\subseteq\mathbb{R}_n$ is said to be Schur-convex (Schur-concave) on $\mathbb{S}$, if $\boldsymbol{\varsigma}\stackrel{m}{\succcurlyeq}\boldsymbol{\varepsilon}$ implies $\Xi(\boldsymbol{\varsigma})\geq(\leq)~\Xi(\boldsymbol{\varepsilon})$ for any $\boldsymbol{\varsigma}$, $\boldsymbol{\varepsilon}\in\mathbb{S}$.
	\end{definition}

	\begin{definition}\label{definition2.4}
		Let $P=\{p_{ij}\}$ and $Q=\{q_{ij}\}$ be two $m\times n$ matrices. Further, let $p_1^R,\ldots,p_m^R$ and $q_1^R,\ldots,q_m^R$ be two rows of $P$ and $Q$, respectively in such a way that each of these quantities is a row vector of length $n$. Then, $P$ is said to be chain majorized by $Q$ (abbreviated as $P\gg Q$) if there exists a finite number of $n\times n$ $T$-transform matrices $T_{\omega_1},\ldots,T_{\omega_k}$ such that $Q=PT_{\omega_1}\ldots T_{\omega_k}$.
	\end{definition}
	
	A $T$-transform matrix has the form $T=\varpi I+(1-\varpi)\Pi$, where $0\leq\varpi\leq1$ and $\Pi$ is a permutation matrix that just interchanges two coordinates that is, row and column.
	Define a matrix
	\begin{eqnarray*}
		\mathcal{L}_n=\left\lbrace
		(\boldsymbol{\varrho},\boldsymbol{\tau})=
		\begin{bmatrix}
			\varrho_{1} & \ldots & \varrho_{n}\\
			\tau_{1} & \ldots & \tau_{n}
		\end{bmatrix}:~\varrho_{i},\tau_{j}>0,~(\varrho_{i}-\varrho_{j})(\tau_{i}-\tau_{j})\leq0,~\forall~i,~j=1,\ldots,n
		\right\rbrace.
	\end{eqnarray*}
	
	\begin{lemma}\label{lemma2.1}
		A differentiable function $\Upsilon:\mathbb{R}_4^+\rightarrow\mathbb{R}^+$ satisfies $\Upsilon(P)\geq\Upsilon(Q)$ for all $P$, $Q$ such that $P\in\mathcal{L}_2$, and $P\gg Q$ if and only if
		\begin{itemize}
			\item[(i)] $\Upsilon(P)=\Upsilon(P\Pi)$ for all permutation matrices $\Pi$, and for all $P\in\mathcal{L}_2$ and;
			
			\item[(ii)] $\sum_{i=1}^{2}(p_{ik}-p_{ij})[\Upsilon_{ik}(P)-\Upsilon_{ij}(P)]\geq 0$,~$\forall$ $j,k=1,2$ and for all $P\in\mathcal{L}_2$, where $\Upsilon_{ij}(P)=\frac{\partial\Upsilon(P)}{\partial p_{ij}}$.
		\end{itemize}
	\end{lemma}

	\section{Model description and results}\label{section3}
	Consider $n$ number of homogeneous independent subpopulations of items, denoted by $\pi_1,\ldots,\pi_n$, which are infinite in nature. Let $X_i$ be the lifetime of an item of $\pi_i$, $i=1,\ldots,n$. Further, let $R(\boldsymbol{X},\boldsymbol{p})$ denote the RV, representing the mixture of items taken from $\pi_1,\ldots,\pi_n$, where $\boldsymbol{X}=(X_1,\ldots,X_n)$ and $\boldsymbol{p}=(p_1,\ldots,p_n)$. Here, $p_i$'s $(>0)$ are known as the mixing proportions with $\sum_{i=1}^{n}p_i=1$. Thus, the survival function of the mixture random variable (MRV) $R(\boldsymbol{X},\boldsymbol{p})$ is given by (see \citealp{McLachlan2000})
	\begin{eqnarray}\label{eq-3.5}
		\bar{F}_{R(\boldsymbol{X},\boldsymbol{p})}(x)=\sum_{i=1}^{n}p_i\bar{F}_{X_i}(x),
	\end{eqnarray}
	where $\bar{F}_{X_i}(x)$ is the survival function of the items of $\pi_i$. For more details on formulation of mixture distribution, readers are referred to \cite{chen2017finite}, \cite{mclachlan2019finite}, and \cite{McLachlan2000}. Because failure rate is a conditional characteristic, the equivalent mixture failure rate is defined using modified conditional weights (on the condition of survival function in $x\geq 0$). For detailed discussion on mixture failure rate, one may refer to the references \cite{navarro2004obtain}, \cite{finkelstein2008failure}, \cite{cha2013failure}, and so on. In this paper, we  consider that the lifetime of the unit of $i$-th subpopulation, denoted by $X_{\alpha_i,\beta_i}$ follows $\mathcal{IK}(\alpha_i,\beta_i)$ distribution, $i=1,\ldots,n$. Denote by $R_n(\boldsymbol{X}_{\boldsymbol{\alpha},\boldsymbol{\beta}};\boldsymbol{p})$ the MRV, constructing from $\pi_1,\ldots,\pi_n$ with $i$-th subpopulation $\pi_i$, following $\mathcal{IK}(\alpha_i,\beta_i)$ distribution, where $\boldsymbol{X}_{\boldsymbol{\alpha},\boldsymbol{\beta}}=(X_{\alpha_1,\beta_1},\ldots,X_{\alpha_n,\beta_n})$. Thus, from (\ref{eq-3.5}), we write
	\begin{eqnarray}\label{eq-3.6}
		\bar{F}_{R_n(\boldsymbol{X}_{\boldsymbol{\alpha},\boldsymbol{\beta}};\boldsymbol{p})}(x)=\sum_{i=1}^{n}p_i[1-(1-(1+x)^{-\alpha_i})^{\beta_i}],~x>0,
	\end{eqnarray}
	where $\boldsymbol{\alpha}=(\alpha_1,\ldots,\alpha_n)$ and $\boldsymbol{\beta}=(\beta_1,\ldots,\beta_n)$.
	
	In this section, we obtain ordering results between two FMMs under three different scenarios. In particular, we consider heterogeneity in one parameter, two parameters, and three parameters.
	
	\subsection{Ordering results for MMs when heterogeneity presents in one parameter}\label{subsection3.1}
	In this subsection, we obtain stochastic ordering results, by considering heterogeneity in one parameter. The first result studies usual stochastic ordering between two FMMs, when heterogeneity is present in mixing proportions. Furthermore, in this model, we consider $\boldsymbol{\alpha}$ as a common shape parameter vector, and $\beta$ as a fixed shape parameter. The following lemma is useful in proving the upcoming theorem.
	
	\begin{lemma}\label{lemma3.1}
		The function $\eta(x;\alpha,\beta)=1-(1-(1+x)^{-\alpha})^{\beta},~x>0,~\alpha>0,~\beta>0$,
		\begin{itemize}
			\item[(i)] is decreasing and convex with respect to $\alpha>0$, for fixed $x>0$, $\beta>0$;
			\item[(ii)] is increasing with respect to $\beta>0$, for fixed $x>0$, $\alpha>0$.
		\end{itemize}
	\end{lemma}
	
	\begin{proof}
		The proof is straightforward, and hence it is omitted.
	\end{proof}
	
	\begin{assumption}\label{assumption3.1}
		Let $\pi_1,\ldots,\pi_n$ be $n$ homogeneous and independent (infinite) subpopulations, with lifetime (denoted by $X_i$) of the items in $i$-th subpopulation $\pi_i$ following $\mathcal{IK}(\alpha_i,\beta_i)$ distribution, $i=1,\ldots,n$.
	\end{assumption}
	
	\begin{theorem}\label{theorem3.1}
		Under the setup in Assumption \ref{assumption3.1}, with $\beta_1=\ldots=\beta_n=\beta$, for
		$(\boldsymbol{\alpha},\boldsymbol{p})$, $(\boldsymbol{\alpha},\boldsymbol{p^*})\in\mathcal{L}_n$, and $\beta>0$, we have
		\begin{eqnarray*}
			\boldsymbol{p^*}\preccurlyeq_w\boldsymbol{p}\Rightarrow
			R_n(\boldsymbol{X}_{\boldsymbol{\alpha},\beta};\boldsymbol{p^*})\leq_{st}
			R_n(\boldsymbol{X}_{\boldsymbol{\alpha},\beta};\boldsymbol{p}),
		\end{eqnarray*}
		where $\boldsymbol{X}_{\boldsymbol{\alpha},\beta}=(X_{\alpha_1,\beta},\ldots,X_{\alpha_n,\beta})$.
	\end{theorem}
	
	\begin{proof}
		See Appendix.
	\end{proof}
	
	The following corollary is an immediate consequence of Theorem \ref{theorem3.1} using the well-known result $\stackrel{m}{\preccurlyeq}\Rightarrow\preccurlyeq_w$.
	
	\begin{corollary}\label{corollary3.1}
		Based on the assumptions and conditions as in Theorem \ref{theorem3.1}, we have
		\begin{eqnarray*}
			\boldsymbol{p^*}\stackrel{m}{\preccurlyeq}\boldsymbol{p}\Rightarrow R_n(\boldsymbol{X}_{\boldsymbol{\alpha},\beta};\boldsymbol{p^*})\leq_{st}R_n(\boldsymbol{X}_{\boldsymbol{\alpha},\beta};\boldsymbol{p}).
		\end{eqnarray*}
	\end{corollary}
	
	To validate Theorem \ref{theorem3.1}, we next provide an example.
	\begin{example}\label{example3.1}
		For $n=3$, consider $\boldsymbol{\alpha}=(0.5,0.4,0.3)$, $\boldsymbol{p}=(0.2,0.2,0.6)$, $\boldsymbol{p^*}=(0.3,0.3,0.4)$, and $\beta=1$. Clearly, the sufficient conditions in Theorem \ref{theorem3.1} are satisfied. Based on these numerical values, the sfs of $R_3(\boldsymbol{X}_{\boldsymbol{\alpha},\beta};\boldsymbol{p^*})$ and $R_3(\boldsymbol{X}_{\boldsymbol{\alpha},\beta};\boldsymbol{p})$ are plotted in Figure \ref{figure1(a)}, validating the result established in Theorem \ref{theorem3.1}.
	\end{example}
	
	Below, we consider a counterexample to reveal that the usual stochastic order between two MRVs may not hold if $(\boldsymbol{\alpha},\boldsymbol{p})$ and $(\boldsymbol{\alpha},\boldsymbol{p^*})$ do not belong to $\mathcal{L}_3$.
	
	\begin{counterexample}\label{counterexample3.1}
		With $n=3$, let $\boldsymbol{p}=(0.1,0.7,0.2)$, $\boldsymbol{p^*}=(0.2,0.5,0.3)$, $\boldsymbol{\alpha}=(3.5,4.8,5.6)$, and $\beta=20$. Here, $(0.2,0.5,0.3)\preccurlyeq_w(0.1,0.7,0.2)$, however, other two conditions are not satisfied. That is, $(\boldsymbol{\alpha},\boldsymbol{p}),(\boldsymbol{\alpha},\boldsymbol{p^*})\notin\mathcal{L}_3$. Now, we plot $\bar{F}_{R_3(\boldsymbol{X}_{\boldsymbol{\alpha},\beta};\boldsymbol{p^*})}(x)-\bar{F}_{R_3(\boldsymbol{X}_{\boldsymbol{\alpha},\beta};\boldsymbol{p})}(x)$ in Figure \ref{figure1(b)}. From the graph, it is clear that the desired usual stochastic order between $R_3(\boldsymbol{X}_{\boldsymbol{\alpha},\beta};\boldsymbol{p^*})$ and $R_3(\boldsymbol{X}_{\boldsymbol{\alpha},\beta};\boldsymbol{p})$ does not hold.
	\end{counterexample}
	
	\begin{figure}
		\begin{center}
			\subfigure[]{\label{figure1(a)}\includegraphics[width=3.2in]{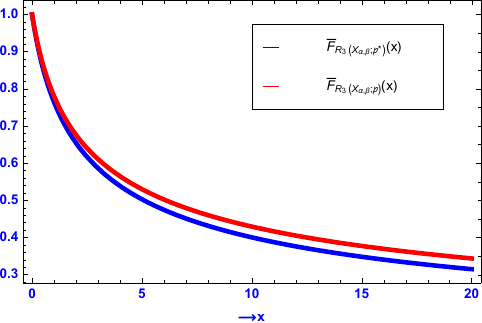}}
			\subfigure[]{\label{figure1(b)}\includegraphics[width=3.2in,height=2.18in]{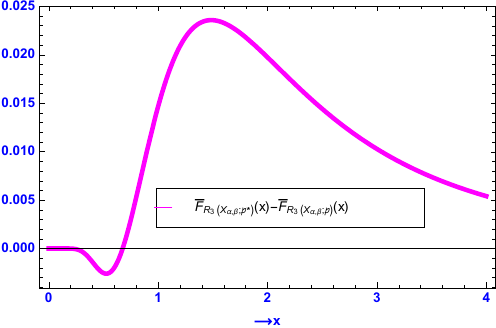}}
			\caption{(a) Plots of the sfs of $R_3(\boldsymbol{X}_{\boldsymbol{\alpha},\beta};\boldsymbol{p^*})$ (blue curve) and $R_3(\boldsymbol{X}_{\boldsymbol{\alpha},\beta};\boldsymbol{p})$ (red curve) in Example \ref{example3.1}.
				(b) Plot of the difference between the sfs of $R_3(\boldsymbol{X}_{\boldsymbol{\alpha},\beta};\boldsymbol{p^*})$ and $R_3(\boldsymbol{X}_{\boldsymbol{\alpha},\beta};\boldsymbol{p})$ in Counterexample \ref{counterexample3.1}.}
		\end{center}
	\end{figure}
	
	Denote $\boldsymbol{X}_{\alpha,\boldsymbol{\beta}}=(X_{\alpha,\beta_1},\ldots,X_{\alpha,\beta_n})$. In the next result, we obtain usual stochastic ordering between $R_n(\boldsymbol{X}_{\alpha,\boldsymbol{\beta}};\boldsymbol{p^*})$
	and $R_n(\boldsymbol{X}_{\alpha,\boldsymbol{\beta}};\boldsymbol{p})$, taking heterogeneity in mixing proportions. Additionally, we consider common shape parameter vector $\boldsymbol{\beta}$ and fixed $\alpha$.
	
	\begin{theorem}\label{theorem3.2}
		Under the setup in Assumption \ref{assumption3.1}, with $\alpha_1=\ldots=\alpha_n=\alpha$, for
		$(\boldsymbol{\beta},\boldsymbol{p}),
		(\boldsymbol{\beta},\boldsymbol{p^*})\in\mathcal{L}_n$, and $\alpha>0$, we have
		\begin{eqnarray*}
			\boldsymbol{p^*}\stackrel{w}{\preccurlyeq}\boldsymbol{p}\Rightarrow
			R_n(\boldsymbol{X}_{\alpha,\boldsymbol{\beta}};\boldsymbol{p^*})\geq_{st}
			R_n(\boldsymbol{X}_{\alpha,\boldsymbol{\beta}};\boldsymbol{p}).
		\end{eqnarray*}
	\end{theorem}
	
	\begin{proof}
		See Appendix.
	\end{proof}

	The following example illustrates Theorem \ref{theorem3.2}, for $n=3$.
	
	\begin{example}\label{example3.2}
		Assume that $\boldsymbol{p}=(0.1,0.3,0.6)$, $\boldsymbol{p^*}=(0.2,0.3,0.5)$, $\boldsymbol{\beta}=(0.5,0.4,0.3)$, and $\alpha=0.5$. Clearly, the assumptions made in Theorem \ref{theorem3.2} are satisfied. Using the numerical values, we plot the graphs of $R_3(\boldsymbol{X}_{\alpha,\boldsymbol{\beta}};\boldsymbol{p^*})$ and $R_3(\boldsymbol{X}_{\alpha,\boldsymbol{\beta}};\boldsymbol{p})$ in Figure \ref{figure2(a)}, validating the usual stochastic order in Theorem \ref{theorem3.2}.
	\end{example}

	Next, we consider a counterexample to establish that the result in Theorem \ref{theorem3.2} does not hold if $(\boldsymbol{\beta},\boldsymbol{p})$ and $(\boldsymbol{\beta},\boldsymbol{p^*})$ do not belong to $\mathcal{L}_3$.

	\begin{counterexample}\label{counterexample3.2}
		For $n=3$, set $\boldsymbol{p}=(0.2,0.6,0.2)$, $\boldsymbol{p^*}=(0.2,0.5,0.3)$, $\boldsymbol{\beta}=(5.2,15.8,5.6)$, and $\alpha=1$. Here, though $\boldsymbol{p^*}\stackrel{w}{\preccurlyeq}\boldsymbol{p}$ holds, the conditions $(\boldsymbol{\beta},\boldsymbol{p})\in\mathcal{L}_3$ and $(\boldsymbol{\beta},\boldsymbol{p^*})\in\mathcal{L}_3$ do not hold. Write
		\begin{eqnarray}\label{eq-3.7}
			K_1(x)&=&\bar{F}_{R_3(\boldsymbol{X}_{\alpha,\boldsymbol{\beta}};\boldsymbol{p^*})}(x)-\bar{F}_{R_3(\boldsymbol{X}_{\alpha,\boldsymbol{\beta}};\boldsymbol{p})}(x)\nonumber\\
			&=&\sum_{i=1}^{3}p_i^*[1-(1-(1+x)^{-1})^{\beta_i}]-\sum_{i=1}^{3}p_i[1-(1-(1+x)^{-1})^{\beta_i}].
		\end{eqnarray}
		Now, for $x=10$, $K_1(10)=0.00262105~(>0)$ and for $x=100$, $K_1(100)=-0.00408561~(<0)$, establishing that $K_1(x)$ changes its sign. Thus, clearly the usual stochastic order between $R_3(\boldsymbol{X}_{\alpha,\boldsymbol{\beta}};\boldsymbol{p^*})$ and $R_3(\boldsymbol{X}_{\alpha,\boldsymbol{\beta}};\boldsymbol{p})$ does not hold.
	\end{counterexample}
	
	\begin{figure}
		\begin{center}
			\subfigure[]{\label{figure2(a)}\includegraphics[width=3.2in]{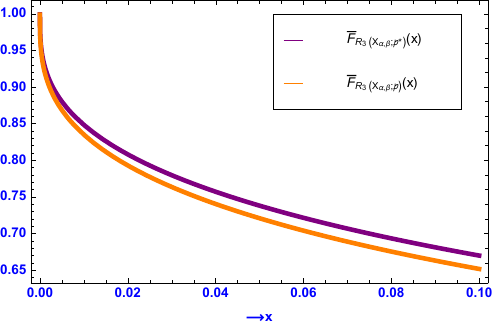}}
			\subfigure[]{\label{figure2(b)}\includegraphics[width=3.2in]{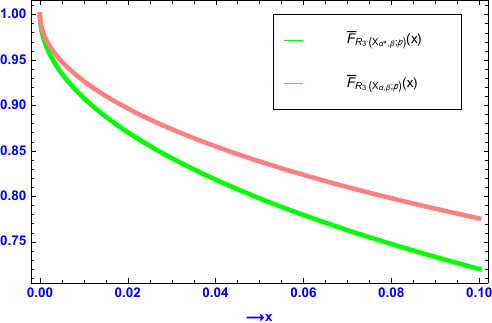}}
			\caption{(a) Plots of the sfs of $R_3(\boldsymbol{X}_{\alpha,\boldsymbol{\beta}};\boldsymbol{p^*})$ (purple curve) and $R_3(\boldsymbol{X}_{\alpha,\boldsymbol{\beta}};\boldsymbol{p})$ (orange curve) in Example \ref{example3.2}.
				(b) Plots of the sfs of $R_3(\boldsymbol{X}_{\boldsymbol{\alpha^*},\beta};\boldsymbol{p})$ (green curve) and $R_3(\boldsymbol{X}_{\boldsymbol{\alpha},\beta};\boldsymbol{p})$ (pink curve) in Example \ref{example3.3}. }
		\end{center}
	\end{figure}

	In the upcoming theorem, we consider heterogeneity in the shape parameter $\alpha$, while $\beta$ is fixed. In addition, we assume the mixing proportion vector $\boldsymbol{p}$ to be common. Denote $\boldsymbol{X}_{\boldsymbol{\alpha^*},\beta}=(X_{\alpha_1^*,\beta},\ldots,X_{\alpha_n^*,\beta})$.
	
	\begin{theorem}\label{theorem3.3}
		Under the setup as in Assumption \ref{assumption3.1}, with $\beta_1=\ldots=\beta_n=\beta~(0<\beta<1)$, for
		$(\boldsymbol{\alpha},\boldsymbol{p}), (\boldsymbol{\alpha^*},\boldsymbol{p})\in\mathcal{L}_n$, we
		have
		\begin{eqnarray*}
			\boldsymbol{\alpha^*}
			\stackrel{w}{\preccurlyeq}\boldsymbol{\alpha}\Rightarrow
			R_n(\boldsymbol{X}_{\boldsymbol{\alpha^*},\beta};\boldsymbol{p})\leq_{st}
			R_n(\boldsymbol{X}_{\boldsymbol{\alpha},\beta};\boldsymbol{p}).
		\end{eqnarray*}
	\end{theorem}
	
	\begin{proof}
		See Appendix.
	\end{proof}
	
	The following example illustrates Theorem \ref{theorem3.3}.
	
	\begin{example}\label{example3.3}
		Consider $n=3$, $\boldsymbol{\alpha}=(1.1,0.9,0.4)$, $\boldsymbol{\alpha^*}=(1.2,0.9,0.8)$, $\boldsymbol{p}=(0.1,0.2,0.7)$ and $\beta=0.5$. Clearly, the assumptions made in Theorem \ref{theorem3.3} are satisfied. Using the numerical values, we plot the graphs of $\bar{F}_{R_3(\boldsymbol{X}_{\boldsymbol{\alpha^*},\beta};\boldsymbol{p})}(x)$ and $\bar{F}_{R_3(\boldsymbol{X}_{\boldsymbol{\alpha},\beta};\boldsymbol{p})}(x)$ in Figure \ref{figure2(b)}, validating the usual stochastic order in Theorem \ref{theorem3.3}.
	\end{example}
	
	Next, we present a counterexample to show that the condition $\boldsymbol{\alpha^*}\stackrel{w}{\preccurlyeq}\boldsymbol{\alpha}$ in Theorem \ref{theorem3.3} cannot be dropped to get the usual stochastic ordering between the MRVs.
	
	\begin{counterexample}\label{counterexample3.3}
		Assume that $n=3$, $\boldsymbol{\alpha}=(1.2,0.8,0.7)$,  $\boldsymbol{\alpha^*}=(1.5,0.9,0.55)$, $\boldsymbol{p}=(0.20,0.35,0.45)$ and $\beta=0.5$. Here, $(1.5,0.9,0.55)\not\stackrel{w}{\preccurlyeq}(1.2,0.8,0.7)$, while other conditions are satisfied. We plot the difference $\bar{F}_{R_3(\boldsymbol{X}_{\boldsymbol{\alpha^*},\beta};\boldsymbol{p})}(x)-\bar{F}_{R_3(\boldsymbol{X}_{\boldsymbol{\alpha},\beta};\boldsymbol{p})}(x)$ in Figure \ref{figure3(a)}, which suggests that the usual stochastic order between $R_3(\boldsymbol{X}_{\boldsymbol{\alpha^*},\beta};\boldsymbol{p})$ and $R_3(\boldsymbol{X}_{\boldsymbol{\alpha},\beta};\boldsymbol{p})$ does not hold.
	\end{counterexample}

	\begin{figure}
		\begin{center}
			\subfigure[]{\label{figure3(a)}\includegraphics[width=3.2in,height=2.15in]{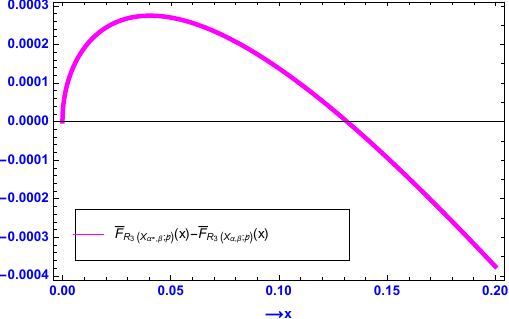}}
			\subfigure[]{\label{figure3(b)}\includegraphics[width=3.2in]{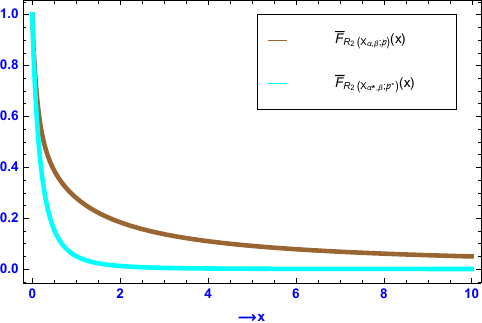}}
			\caption{(a) Plot of the difference between the sfs of $R_3(\boldsymbol{X}_{\boldsymbol{\alpha^*},\beta};\boldsymbol{p})$ and $R_3(\boldsymbol{X}_{\boldsymbol{\alpha},\beta};\boldsymbol{p})$ in Counterexample \ref{counterexample3.3}.
				(b) Plots of the sfs of $R_2(\boldsymbol{X}_{\boldsymbol{\alpha},\beta};\boldsymbol{p})$ (brown curve) and $R_2(\boldsymbol{X}_{\boldsymbol{\alpha^*},\beta};\boldsymbol{p^*})$ (cyan curve) in Example \ref{example3.4}.}
		\end{center}
	\end{figure}

	\subsection{Ordering results for MMs when heterogeneity presents in two parameters}\label{subsection3.2}
	
	In the previous subsection, we have assumed heterogeneity in one parameter. There are various situations, where more than one parameter is heterogeneous. In this subsection, we consider heterogeneity in two parameters and obtain some ordering results. The results have been established using the concept of chain majorization between the parameter-matrices of two MMs. The following lemma is useful in proving the upcoming theorem.
	
	\begin{lemma}\label{lemma3.2}
		The function $\zeta(x;\alpha,\beta)=(1+x)^{-\alpha}(1-(1+x)^{-\alpha})^{\beta-1}$, for fixed $x>0$ and $0<\beta<1$, is decreasing with respect to $\alpha>0$.
	\end{lemma}
	
	\begin{proof}
		The proof is straightforward, and thus it is omitted.
	\end{proof}
	
	\begin{theorem}\label{theorem3.4}
		Let
		$\bar{F}_{R_2(\boldsymbol{X}_{\boldsymbol{\alpha},\beta};\boldsymbol{p})}(x)
		=\sum_{i=1}^{2}p_i[1-(1-(1+x)^{-\alpha_i})^\beta]$ and
		$\bar{F}_{R_2(\boldsymbol{X}_{\boldsymbol{\alpha^*},\beta};\boldsymbol{p^*})}(x)
		=\sum_{i=1}^{2}p_i^*[1-(1-(1+x)^{-\alpha_i^*})^\beta]$ be the sfs of the MRVs $R_2(\boldsymbol{X}_{\boldsymbol{\alpha},\beta};\boldsymbol{p})$ and $R_2(\boldsymbol{X}_{\boldsymbol{\alpha^*},\beta};\boldsymbol{p^*})$, respectively.    For $(\boldsymbol{p},\boldsymbol{\alpha})\in\mathcal{L}_2$
		and $\beta\in(0,1)$, we have
		\begin{eqnarray*}
			\begin{bmatrix}
				p_1 & p_2\\
				\alpha_1 & \alpha_2
			\end{bmatrix}\gg
			\begin{bmatrix}
				p_1^* & p_2^*\\
				\alpha_1^* & \alpha_2^*
			\end{bmatrix}\Rightarrow R_2(\boldsymbol{X}_{\boldsymbol{\alpha},\beta};\boldsymbol{p})\geq_{st} R_2(\boldsymbol{X}_{\boldsymbol{\alpha^*},\beta};\boldsymbol{p^*}).
		\end{eqnarray*}
	\end{theorem}
	
	\begin{proof}
		See Appendix.
	\end{proof}
	
	To validate Theorem \ref{theorem3.4}, we consider the following example.
	
	\begin{example}\label{example3.4}
		For $n=2$, consider $\boldsymbol{p}=(0.6,0.4)$, $\boldsymbol{p^*}=(0.46,0.54)$, $\boldsymbol{\alpha}=(1,9)$, $\boldsymbol{\lambda}=(6.6,3.4)$, and $\beta=0.5$. Clearly, we observe that
		$(\begin{smallmatrix}
		0.6 & 0.4\\
		1 & 9
		\end{smallmatrix})\in\mathcal{L}_2$. Take a $T$-transform matrix
		$T_{0.3}=(\begin{smallmatrix}
		0.3 & 0.7\\
		0.7 & 0.3
		\end{smallmatrix})$. It can then be seen that
		\begin{eqnarray*}
			\begin{pmatrix}
				0.46 & 0.54\\
				6.6 & 3.4
			\end{pmatrix}=
			\begin{pmatrix}
				0.6 & 0.4\\
				1 & 9
			\end{pmatrix}\times
			\begin{pmatrix}
				0.3 & 0.7\\
				0.7 & 0.3
			\end{pmatrix}.
		\end{eqnarray*}
		Thus,
		\begin{eqnarray*}
			\begin{pmatrix}
				0.6 & 0.4\\
				1 & 9
			\end{pmatrix}\gg
			\begin{pmatrix}
				0.46 & 0.54\\
				6.6 & 3.4
			\end{pmatrix},
		\end{eqnarray*}
		and from Theorem \ref{theorem3.4}, we obtain $R_2(\boldsymbol{X}_{\boldsymbol{\alpha},\beta};\boldsymbol{p})\geq_{st}R_2(\boldsymbol{X}_{\boldsymbol{\alpha^*},\beta};\boldsymbol{p^*})$, as observed in Figure \ref{figure3(b)}.
	\end{example}
	
	Note that Theorem \ref{theorem3.4} has been established for $\beta\in(0,1)$ with other sufficient conditions. Thus, one may be curious to know whether the usual stochastic order in Theorem \ref{theorem3.4} holds if $\beta$ is not in $(0,1)$. In this regard, we consider the next counterexample.
	
	\begin{counterexample}\label{counterexample3.4}
		For $n=2$, set $\boldsymbol{\alpha}=(2,3)$, $\boldsymbol{p}=(0.7,0.3)$, $\boldsymbol{\alpha^*}=(2.6,2.4)$, $\boldsymbol{p^*}=(0.46,0.54)$, and $\beta=100$. Clearly, $\beta\notin(0,1)$. Further, for a $T$-transform matrix $T_{0.4}=(\begin{smallmatrix}
		0.4 & 0.6\\
		0.6 & 0.4
		\end{smallmatrix})$, we obtain
		\begin{eqnarray*}
			\begin{pmatrix}
				0.46 & 0.54\\
				2.6 & 2.4
			\end{pmatrix}=
			\begin{pmatrix}
				0.7 & 0.3\\
				2 & 3
			\end{pmatrix}\times
			\begin{pmatrix}
				0.4 & 0.6\\
				0.6 & 0.4
			\end{pmatrix},
		\end{eqnarray*}
		implying that
		\begin{eqnarray*}
			\begin{pmatrix}
				0.7 & 0.3\\
				2 & 3
			\end{pmatrix}\gg
			\begin{pmatrix}
				0.46 & 0.54\\
				2.6 & 2.4
			\end{pmatrix}.
		\end{eqnarray*}
		Now, we plot the graph of the difference of the sfs of $R_2(\boldsymbol{X}_{\boldsymbol{\alpha},\beta};\boldsymbol{p})$ and $R_2(\boldsymbol{X}_{\boldsymbol{\alpha^*},\beta};\boldsymbol{p^*})$ in Figure \ref{figure4(a)}. From this figure, the usual stochastic order between $R_2(\boldsymbol{X}_{\boldsymbol{\alpha},\beta};\boldsymbol{p})$ and $R_2(\boldsymbol{X}_{\boldsymbol{\alpha^*},\beta};\boldsymbol{p^*})$ does not hold, since the graph cross $x$-axis.
	\end{counterexample}
	
	Next, we extend the result in Theorem \ref{theorem3.4} for $n$ number of subpopulations.
	
	\begin{theorem}\label{theorem3.5}
		Let $\bar{F}_{R_n(\boldsymbol{X}_{\boldsymbol{\alpha},\beta};\boldsymbol{p})}(x)=\sum_{i=1}^{n}p_i[1-(1-(1+x)^{-\alpha_i})^{\beta}]$ and  $\bar{F}_{R_n(\boldsymbol{X}_{\boldsymbol{\alpha^*},\beta};\boldsymbol{p^*})}(x)=\sum_{i=1}^{n}p_i^*[1-(1-(1+x)^{-\alpha_i^*})^{\beta}]$ be the sfs of the MRVs $R_n(\boldsymbol{X}_{\boldsymbol{\alpha},\beta};\boldsymbol{p})$ and $R_n(\boldsymbol{X}_{\boldsymbol{\alpha^*},\beta};\boldsymbol{p^*})$, respectively. For
		$(\boldsymbol{p},\boldsymbol{\alpha})\in\mathcal{L}_n$ and $\beta\in(0,1)$,
		\begin{eqnarray*}
			\left(\begin{smallmatrix}
				p_1\ldots p_n\\
				\alpha_1\ldots \alpha_n
			\end{smallmatrix}\right)\gg
			\left(\begin{smallmatrix}
				p_1^*\ldots p_n^*\\
				\alpha_1^*\ldots \alpha_n^*
			\end{smallmatrix}\right)\Rightarrow R_n(\boldsymbol{X}_{\boldsymbol{\alpha},\beta};\boldsymbol{p})\geq_{st} R_n(\boldsymbol{X}_{\boldsymbol{\alpha^*},\beta};\boldsymbol{p^*}).
		\end{eqnarray*}
	\end{theorem}
	
	\begin{proof}
		The proof of this theorem is similar to that of Theorem $4.2$ of \cite{bhakta2023stochastic}. Thus, it is omitted.
	\end{proof}
	
	It is well-known fact that a finite product of $T$-transform matrices with a common structure yields a $T$-transform matrix. Using this, the following corollary is immediate consequence of Theorem \ref{theorem3.5}.
	
	\begin{corollary}\label{corollary3.2}
		Consider $k$ number of $T$-transformed matrices with a common structure, denoted by $T_{\omega_1},\ldots,T_{\omega_n}$. Then, under the setup in Theorem \ref{theorem3.5}, with $(\boldsymbol{p},\boldsymbol{\alpha})\in\mathcal{L}_n$ and $\beta\in(0,1)$,
		\begin{eqnarray*}
			\left(\begin{smallmatrix}
				p_1\ldots p_n\\
				\alpha_1\ldots \alpha_n
			\end{smallmatrix}\right)\gg
			\left(\begin{smallmatrix}
				p_1^*\ldots p_n^*\\
				\alpha_1^*\ldots \alpha_n^*
			\end{smallmatrix}\right)\Rightarrow R_n(\boldsymbol{X}_{\boldsymbol{\alpha},\beta};\boldsymbol{p})\geq_{st} R_n(\boldsymbol{X}_{\boldsymbol{\alpha^*},\beta};\boldsymbol{p^*}).
		\end{eqnarray*}
	\end{corollary}
	From the result in corollary \ref{corollary3.2}, it is an obvious question that ``does the result in corollary \ref{corollary3.2} hold if the $T$-transform matrices have different structures?'' In the next theorem we discuss this issue. We observe that the similar result to corollary \ref{corollary3.2} holds with an additional assumption.
	
	\begin{theorem}\label{theorem3.6}
		Consider $k$ number of $T$-transformed matrices with different structures, denoted by $T_{\omega_1},\ldots,T_{\omega_k}$. Then, under the similar setup as in Theorem \ref{theorem3.5}, with $(\boldsymbol{p},\boldsymbol{\alpha})\in\mathcal{L}_n$,
		$(\boldsymbol{p},\boldsymbol{\alpha})T_{\omega_1}\ldots T_{\omega_i}\in\mathcal{L}_n$, $i=1,\ldots,k-1,~(k\geq2)$, and $\beta\in(0,1)$,
		\begin{eqnarray*}
			\left(\begin{smallmatrix}
				p_1^*\ldots p_n^*\\
				\alpha_1^*\ldots\alpha_n^*
			\end{smallmatrix}\right)=
			\left(\begin{smallmatrix}
				p_1\ldots p_n\\
				\alpha_1\ldots \alpha_n
			\end{smallmatrix}\right)T_{\omega_1}\ldots T_{\omega_k}\Rightarrow R_n(\boldsymbol{X}_{\boldsymbol{\alpha},\beta};\boldsymbol{p})\geq_{st} R_n(\boldsymbol{X}_{\boldsymbol{\alpha^*},\beta};\boldsymbol{p^*}).
		\end{eqnarray*}
	\end{theorem}
	
	\begin{proof}
		The proof is similar to that of the proof of Theorem $4.3$ of \cite{bhakta2023stochastic}. Hence, it is not presented.
	\end{proof}
	
	\begin{figure}
		\begin{center}
			\subfigure[]{\label{figure4(a)}\includegraphics[width=3.2in,height=2.15in]{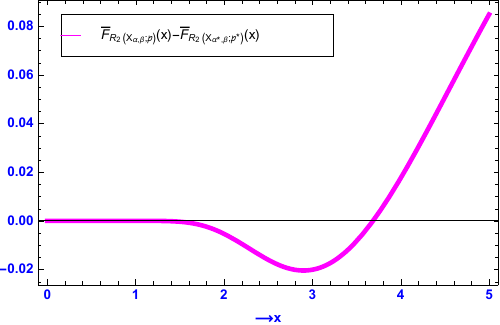}}
			\subfigure[]{\label{figure4(b)}\includegraphics[width=3.2in]{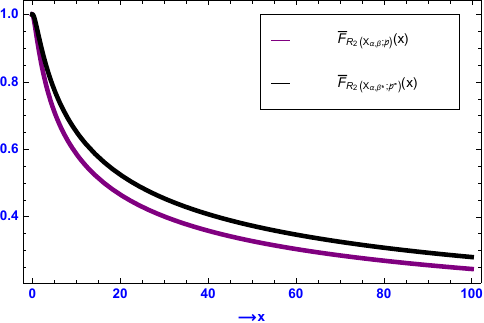}}
			\caption{(a) Plot of the difference between the sfs of $R_2(\boldsymbol{X}_{\boldsymbol{\alpha},\beta};\boldsymbol{p})$ and $R_2(\boldsymbol{X}_{\boldsymbol{\alpha^*},\beta};\boldsymbol{p^*})$ in Counterexample \ref{counterexample3.4}.
				(b) Plots of the sfs of $R_2(\boldsymbol{X}_{\alpha,\boldsymbol{\beta}};\boldsymbol{p})$ (purple curve) and $R_2(\boldsymbol{X}_{\alpha,\boldsymbol{\beta^*}};\boldsymbol{p^*})$ (black curve) in Example \ref{example3.5}.}
		\end{center}
	\end{figure}

	Preceding results of this subsection deal with the stochastic comparison of two MRVs when the matrix $(\boldsymbol{p},\boldsymbol{\alpha})$ changes to another matrix $(\boldsymbol{p^*},\boldsymbol{\alpha^*})$ with fixed $\beta$. In the upcoming results, we assume that the matrix $(\boldsymbol{p},\boldsymbol{\beta})$ changes to $(\boldsymbol{p^*},\boldsymbol{\beta^*})$ with fixed $\alpha$. First, we consider two subpopulations. The following lemma is useful to establish the next theorem.
	
	\begin{lemma}\label{lemma3.3}
		The function $\varkappa(x;p,\alpha,\beta)=p(1-(1+x)^{-\alpha})^{\beta}$
		\begin{itemize}
			\item[(i)] is decreasing with respect to $\beta$ for fixed $p>0$, $\alpha>0$;
			\item[(ii)] is increasing with respect to $p$ for fixed $\beta>0$, $\alpha>0$.
		\end{itemize}
	\end{lemma}
	
	\begin{proof}
		The proof is straightforward, and thus it is omitted.
	\end{proof}
	
	\begin{theorem}\label{theorem3.7}
		Let
		$\bar{F}_{R_2(\boldsymbol{X}_{\alpha,\boldsymbol{\beta}};\boldsymbol{p})}(x)=\sum_{i=1}^{2}p_i[1-(1-(1+x)^{-\alpha})^{\beta_i}]$
		and
		$\bar{F}_{R_2(\boldsymbol{X}_{\alpha,\boldsymbol{\beta^*}};\boldsymbol{p^*})}(x)=\sum_{i=1}^{2}p_i^*[1-(1-(1+x)^{-\alpha})^{\beta_i^*}]$
		be the sfs of the MRVs $R_n(\boldsymbol{X}_{\alpha,\boldsymbol{\beta}};\boldsymbol{p})$ and $R_n(\boldsymbol{X}_{\alpha,\boldsymbol{\beta^*}};\boldsymbol{p^*})$, respectively.  For
		$(\boldsymbol{p},\boldsymbol{\beta})\in\mathcal{L}_2$ and fixed $\alpha>0$,
		\begin{eqnarray*}
			\begin{bmatrix}
				p_1 & p_2\\
				\beta_1 & \beta_2
			\end{bmatrix}\gg
			\begin{bmatrix}
				p_1^* & p_2^*\\
				\beta_1^* & \beta_2^*
			\end{bmatrix}\Rightarrow R_2(\boldsymbol{X}_{\alpha,\boldsymbol{\beta}};\boldsymbol{p})\leq_{st} R_2(\boldsymbol{X}_{\alpha,\boldsymbol{\beta^*}};\boldsymbol{p^*}).
		\end{eqnarray*}
	\end{theorem}
	
	\begin{proof}
		See Appendix.
	\end{proof}
	
	In order to justify Theorem \ref{theorem3.7}, an example is provided.
	
	\begin{example}\label{example3.5}
		Let $\boldsymbol{p}=(0.2,0.8)$, $\boldsymbol{p^*}=(0.74,0.26)$, $\boldsymbol{\beta}=(6,2)$, $\boldsymbol{\beta^*}=(2.4,5.6)$, and $\alpha=0.6$. It is not hard to check that $(\boldsymbol{p},\boldsymbol{\beta})\in\mathcal{L}_2$. Further, let
		$T_{0.1}=\left(\begin{smallmatrix}
		0.1 & 0.9\\
		0.9 & 0.1
		\end{smallmatrix}\right)$. It can then be seen that
		\begin{eqnarray*}
			\begin{pmatrix}
				0.74 & 0.26\\
				2.4 & 5.6
			\end{pmatrix}=
			\begin{pmatrix}
				0.2 & 0.8\\
				6 & 2
			\end{pmatrix}\times
			\begin{pmatrix}
				0.1 & 0.9\\
				0.9 & 0.1
			\end{pmatrix},
		\end{eqnarray*}
		which implies that
		\begin{eqnarray*}
			\begin{pmatrix}
				0.2 & 0.8\\
				6 & 2
			\end{pmatrix}\gg
			\begin{pmatrix}
				0.74 & 0.26\\
				2.4 & 5.6
			\end{pmatrix}.
		\end{eqnarray*}
		Thus, from Theorem \ref{theorem3.7}, we obtain $R_2(\boldsymbol{X}_{\alpha,\boldsymbol{\beta}};\boldsymbol{p})\leq_{st} R_2(\boldsymbol{X}_{\alpha,\boldsymbol{\beta^*}};\boldsymbol{p^*})$, which can be verified from Figure \ref{figure4(b)}.
	\end{example}
	
	The following counterexample shows that the desired ordering result in Theorem \ref{theorem3.7} does not hold if $(\boldsymbol{p},\boldsymbol{\beta})\notin\mathcal{L}_2$.
	
	\begin{counterexample}\label{counterexample3.5}
		Let us assume
		\begin{eqnarray*}
			\begin{pmatrix}
				p_1 & p_2\\
				\beta_1 & \beta_2
			\end{pmatrix}=
			\begin{pmatrix}
				0.6 & 0.4\\
				7 & 1
			\end{pmatrix}\notin\mathcal{L}_2~and~
			\begin{pmatrix}
				p_1^* & p_2^*\\
				\beta_1^* & \beta_2^*
			\end{pmatrix}=
			\begin{pmatrix}
				0.48 & 0.52\\
				3.4 & 4.6
			\end{pmatrix}\notin\mathcal{L}_2.
		\end{eqnarray*}
		It is then easy to check that
		\begin{eqnarray*}
			\begin{pmatrix}
				0.48 & 0.52\\
				3.4 & 4.6
			\end{pmatrix}=
			\begin{pmatrix}
				0.6 & 0.4\\
				7 & 1
			\end{pmatrix}\times
			\begin{pmatrix}
				0.4 & 0.6\\
				0.6 & 0.4
			\end{pmatrix},
		\end{eqnarray*}
		where $T_{0.4}=\left(\begin{smallmatrix}
		0.4 & 0.6\\
		0.6 & 0.4
		\end{smallmatrix}\right)$, implying that
		\begin{eqnarray*}
			\begin{pmatrix}
				0.6 & 0.4\\
				7 & 1
			\end{pmatrix}\gg
			\begin{pmatrix}
				0.48 & 0.52\\
				3.4 & 4.6
			\end{pmatrix}.
		\end{eqnarray*}
		Now, the difference between the sfs of the MRVs $R_2(\boldsymbol{X}_{\alpha,\boldsymbol{\beta}};\boldsymbol{p})$ and $R_2(\boldsymbol{X}_{\alpha,\boldsymbol{\beta^*}};\boldsymbol{p^*})$ is plotted in Figure \ref{figure5(a)}. Clearly, the difference take negative as well as positive values, which means that the desired usual stochastic order in Theorem \ref{theorem3.7} does not hold.
	\end{counterexample}
	
	Next, we present a result, dealing with $n$ number of subpopulations.
	
	\begin{theorem}\label{theorem3.8}
		Let $\bar{F}_{R_n(\boldsymbol{X}_{\alpha,\boldsymbol{\beta}};\boldsymbol{p})}(x)=\sum_{i=1}^{n}p_i[1-(1-(1+x)^{-\alpha})^{\beta_i}]$ and  $\bar{F}_{R_n(\boldsymbol{X}_{\alpha,\boldsymbol{\beta^*}};\boldsymbol{p^*})}(x)=\sum_{i=1}^{n}p_i^*[1-(1-(1+x)^{-\alpha})^{\beta_i^*}]$ be the sfs of the MRVs $R_n(\boldsymbol{X}_{\alpha,\boldsymbol{\beta}};\boldsymbol{p})$ and $R_n(\boldsymbol{X}_{\alpha,\boldsymbol{\beta^*}};\boldsymbol{p^*})$, respectively. For
		$(\boldsymbol{p},\boldsymbol{\beta})\in\mathcal{L}_n$ and fixed $\alpha>0$,
		\begin{eqnarray*}
			\left(\begin{smallmatrix}
				p_1\ldots p_n\\
				\beta_1\ldots \beta_n
			\end{smallmatrix}\right)\gg
			\left(\begin{smallmatrix}
				p_1^*\ldots p_n^*\\
				\beta_1^*\ldots \beta_n^*
			\end{smallmatrix}\right)\Rightarrow R_n(\boldsymbol{X}_{\alpha,\boldsymbol{\beta}};\boldsymbol{p})\leq_{st} R_n(\boldsymbol{X}_{\alpha,\boldsymbol{\beta^*}};\boldsymbol{p^*}).
		\end{eqnarray*}
	\end{theorem}
	
	\begin{proof}
		The proof is similar to that of Theorem \ref{theorem3.5}, and thus it is omitted.
	\end{proof}
	
	The following corollary can be established from Theorem \ref{theorem3.8} using the arguments similar to corollary \ref{corollary3.2}.
	
	\begin{corollary}\label{corollary3.3}
		Consider $k$ number of $T$-transform matrices, denoted by $T_{\omega_1},\ldots,T_{\omega_k}$ having a common structure. Then, under the setup in Theorem \ref{theorem3.8}, for $(\boldsymbol{p},\boldsymbol{\beta})\in\mathcal{L}_n$ and fixed $\alpha>0$,
		\begin{eqnarray*}
			\left(\begin{smallmatrix}
				p_1\ldots p_n\\
				\beta_1\ldots \beta_n
			\end{smallmatrix}\right)\gg
			\left(\begin{smallmatrix}
				p_1^*\ldots p_n^*\\
				\beta_1^*\ldots \beta_n^*
			\end{smallmatrix}\right)\Rightarrow R_n(\boldsymbol{X}_{\alpha,\boldsymbol{\beta}};\boldsymbol{p})\leq_{st} R_n(\boldsymbol{X}_{\alpha,\boldsymbol{\beta^*}};\boldsymbol{p^*}).
		\end{eqnarray*}
	\end{corollary}
	
	The next theorem proves a result associated with $k$ $(\geq 2)$ number of $T$-transformed matrices having different structures.
	
	\begin{theorem}\label{theorem3.9}
		Let $T_{\omega_1},\ldots,T_{\omega_i}$, $i=1,\ldots,k~(k\geq 2)$ be the $k$ number of $T$-transformed matrices with different structures. Then, under the setup as in Theorem \ref{theorem3.8}, for
		$(\boldsymbol{p},\boldsymbol{\beta})\in\mathcal{L}_n$,
		$\left(\begin{smallmatrix}
		p_1\ldots p_n\\
		\beta_1\ldots \beta_n
		\end{smallmatrix}\right)T_{\omega_1}\ldots T_{\omega_i}\in\mathcal{L}_n$, $i=1,\ldots,(k-1)$, and fixed $\alpha>0$,
		\begin{eqnarray*}
			\left(\begin{smallmatrix}
				p_1^*\ldots p_n^*\\
				\beta_1^*\ldots \beta_n^*
			\end{smallmatrix}\right)=
			\left(\begin{smallmatrix}
				p_1\ldots p_n\\
				\beta_1\ldots \beta_n
			\end{smallmatrix}\right)T_{\omega_1}\ldots T_{\omega_k}\Rightarrow R_n(\boldsymbol{X}_{\alpha,\boldsymbol{\beta}};\boldsymbol{p})\leq_{st} R_n(\boldsymbol{X}_{\alpha,\boldsymbol{\beta^*}};\boldsymbol{p^*}).
		\end{eqnarray*}
	\end{theorem}

	\begin{figure}
		\begin{center}
			\subfigure[]{\label{figure5(a)}\includegraphics[width=3.2in]{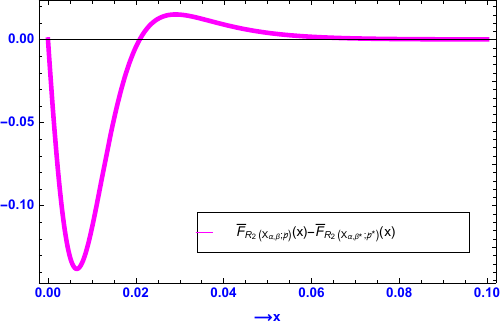}}
			\subfigure[]{\label{figure5(b)}\includegraphics[width=3.09in]{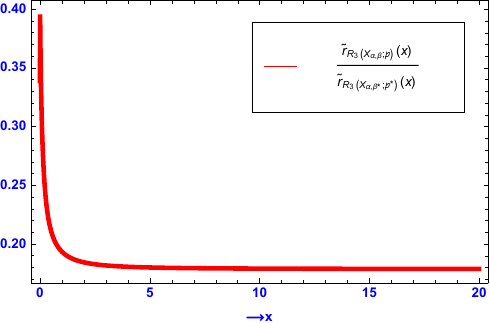}}
			\caption{(a) Plot of the difference between the sfs of $R_2(\boldsymbol{X}_{\alpha,\boldsymbol{\beta}};\boldsymbol{p})$ and $R_2(\boldsymbol{X}_{\alpha,\boldsymbol{\beta^*}};\boldsymbol{p^*})$ in Counterexample \ref{counterexample3.5}.
				(b) Plot of the ratio of the rhs of $R_3(\boldsymbol{X}_{\alpha,\boldsymbol{\beta}};\boldsymbol{p})$ and $R_3(\boldsymbol{X}_{\alpha,\boldsymbol{\beta^*}};\boldsymbol{p^*})$ in Example \ref{example3.6}.}
		\end{center}
	\end{figure}

	We end this subsection with ageing faster order between two MRVs. Here, we assume that mixing proportions and one of the shape parameter vectors are varying. We recall that using ageing faster order one is capable to compare the relative ageings of two engineering systems. In the following theorem, we study the ageing faster order in terms of the reversed hazard rate function.
	
	\begin{theorem}\label{theorem3.10}
		Let $R_n(\boldsymbol{X}_{\alpha,\boldsymbol{\beta}};\boldsymbol{p})$ and $R_n(\boldsymbol{X}_{\alpha,\boldsymbol{\beta^*}};\boldsymbol{p^*})$ be the MRVs with reversed hazard functions     $\tilde{r}_{R_n(\boldsymbol{X}_{\alpha,\boldsymbol{\beta}};\boldsymbol{p})}(x)$ and $\tilde{r}_{R_n(\boldsymbol{X}_{\alpha,\boldsymbol{\beta^*}};\boldsymbol{p^*})}(x)$, respectively. Then,
		\begin{eqnarray*}
			R_n(\boldsymbol{X}_{\alpha,\boldsymbol{\beta}};\boldsymbol{p})\leq_{R-rh}R_n(\boldsymbol{X}_{\alpha,\boldsymbol{\beta^*}};\boldsymbol{p^*}),
		\end{eqnarray*}
		provided $\max\limits_{\{i,j:i\neq j\}}|\beta_i-\beta_j|\leq\min\limits_{\{i,j:i\neq j\}}|\beta_i^*-\beta_j^*|$, for all $i,j=1,\ldots,n$.
	\end{theorem}
	
	\begin{proof}
		See Appendix.
	\end{proof}
	
	The following example illustrates Theorem \ref{theorem3.10}, for $n=3$.

	\begin{example}\label{example3.6}
		Assume that $\boldsymbol{\beta}=(0.1,0.2,0.3)$, $\boldsymbol{\beta^*}=(0.5,1,2)$, $\boldsymbol{p}=(0.1,0.3,0.6)$, $\boldsymbol{p^*}=(0.2,0.3,0.5)$, and $\alpha=2$.   Clearly, the condition $\max|\beta_i-\beta_j|\leq\min|\beta_i^*-\beta_j^*|$ is satisfied for $i\neq j=1,2,3$. Taking these numerical values of the parameters, $\tilde{r}_{R_3(\boldsymbol{X}_{\alpha,\boldsymbol{\beta}};\boldsymbol{p})}(x)/\tilde{r}_{R_3(\boldsymbol{X}_{\alpha,\boldsymbol{\beta^*}};\boldsymbol{p^*})}(x)$ is plotted in Figure \ref{figure5(b)}, validating the result in Theorem \ref{theorem3.10}.
	\end{example}


	\subsection{Ordering results for MMs when heterogeneity presents in three parameters}\label{subsection3.3}
	In the previous subsection, we assume that two parameters are heterogeneous. In this subsection, we present the ordering results considering heterogeneity in three parameters. We mainly obtain the stochastic comparison results between two MRVs with respect to reversed hazard rate and likelihood ratio orders. First, we provide reversed hazard rate order between $R_n(\boldsymbol{X}_{\boldsymbol{\alpha},\boldsymbol{\beta}};\boldsymbol{p})$ and $R_n(\boldsymbol{X}_{\boldsymbol{\alpha^*},\boldsymbol{\beta^*}};\boldsymbol{p^*})$.
	
	\begin{theorem}\label{theorem3.11}
		Consider two MRVs $R_n(\boldsymbol{X}_{\boldsymbol{\alpha},\boldsymbol{\beta}};\boldsymbol{p})$ and $R_n(\boldsymbol{X}_{\boldsymbol{\alpha^*},\boldsymbol{\beta^*}};\boldsymbol{p^*})$ with sfs
		$\bar{F}_{R_n(\boldsymbol{X}_{\boldsymbol{\alpha},\boldsymbol{\beta}};\boldsymbol{p})}(x)$ and
		$\bar{F}_{R_n(\boldsymbol{X}_{\boldsymbol{\alpha^*},\boldsymbol{\beta^*}};\boldsymbol{p^*})}(x)$, respectively. Then,
		\begin{eqnarray*}
			R_n(\boldsymbol{X}_{\boldsymbol{\alpha},\boldsymbol{\beta}};\boldsymbol{p})\leq_{rh}R_n(\boldsymbol{X}_{\boldsymbol{\alpha^*},\boldsymbol{\beta^*}};\boldsymbol{p^*}),
		\end{eqnarray*}
		provided $\max\{\alpha_1^*,\ldots,\alpha_n^*\}\leq\min\{\alpha_1,\ldots,\alpha_n\}$ and $\max\{\alpha_1\beta_1,\ldots,\alpha_n\beta_n\}\leq\min\{\alpha_1^*\beta_1^*,\ldots,\alpha_n^*\beta_n^*\}$.
	\end{theorem}
	
	\begin{proof}
		See Appendix.
	\end{proof}
	
	An illustrative example is provided below, for $n=3$.
	
	\begin{example}\label{example3.7}
		Assume that $\boldsymbol{p}=(0.1,0.7,0.2)$, $\boldsymbol{p^*}=(0.2,0.5,0.3)$, $\boldsymbol{\alpha}=(5,8,6)$, $\boldsymbol{\beta}=(2,1,1)$, $\boldsymbol{\alpha^*}=(3,4,2)$, and $\boldsymbol{\beta^*}=(5,3,6)$. Clearly, $\max\{3,4,2\}\leq\min\{5,8,6\}$ and $\max\{10,8,6\}\leq\min\{15,12,12\}$. Thus, the conditions of Theorem \ref{theorem3.11} hold. Now, the ratio of the cdfs of $R_3(\boldsymbol{X}_{\boldsymbol{\alpha},\boldsymbol{\beta}};\boldsymbol{p})$ and $R_3(\boldsymbol{X}_{\boldsymbol{\alpha^*},\boldsymbol{\beta^*}};\boldsymbol{p^*})$ is plotted in Figure \ref{figure6(a)}, illustrating the result  $R_3(\boldsymbol{X}_{\boldsymbol{\alpha},\boldsymbol{\beta}};\boldsymbol{p})\leq_{rh}R_3(\boldsymbol{X}_{\boldsymbol{\alpha^*},\boldsymbol{\beta^*}};\boldsymbol{p^*})$ in Theorem \ref{theorem3.11}.
	\end{example}
	
	The following counterexample illustrates that the result in Theorem \ref{theorem3.11} does not hold if $\max\{\alpha_1^*,\alpha_2^*,\alpha_3^*\}\nleq\min\{\alpha_1,\alpha_2,\alpha_3\}$.
	
	\begin{counterexample}\label{counterexample3.6}
		Set $\boldsymbol{p}=(0.60,0.25,0.15)$, $\boldsymbol{p^*}=(0.45,0.30,0.25)$,  $\boldsymbol{\alpha}=(1,3,5)$, $\boldsymbol{\beta}=(3,6,9)$, $\boldsymbol{\alpha^*}=(2,4,6)$, and $\boldsymbol{\beta^*}=(25,30,35)$. Obviously,  $\max\{2,4,6\}\nleq\min\{1,3,5\}$ and $\max\{3,18,45\}\leq\min\{50,120,210\}$. It can be seen that all the conditions of Theorem \ref{theorem3.11} are satisfied except $\max\{2,4,6\}\leq\min\{1,3,5\}$. Now, the ratio of the cdfs of the MRVs
		$R_3(\boldsymbol{X}_{\boldsymbol{\alpha},\boldsymbol{\beta}};\boldsymbol{p})$ and $R_3(\boldsymbol{X}_{\boldsymbol{\alpha^*},\boldsymbol{\beta^*}};\boldsymbol{p^*})$ is presented in Figure \ref{figure6(b)}, from which, we see that the ratio is nonmonotone in $x>0$. As a conclusion, $R_3(\boldsymbol{X}_{\boldsymbol{\alpha},\boldsymbol{\beta}};\boldsymbol{p})\nleq_{rh}R_3(\boldsymbol{X}_{\boldsymbol{\alpha^*},\boldsymbol{\beta^*}};\boldsymbol{p^*})$. In other
		words, the reversed hazard rate order between $R_3(\boldsymbol{X}_{\boldsymbol{\alpha},\boldsymbol{\beta}};\boldsymbol{p})$ and $R_3(\boldsymbol{X}_{\boldsymbol{\alpha^*},\boldsymbol{\beta^*}};\boldsymbol{p^*})$ does not hold.
	\end{counterexample}

	\begin{figure}
		\begin{center}
			\subfigure[]{\label{figure6(a)}\includegraphics[width=3.09in]{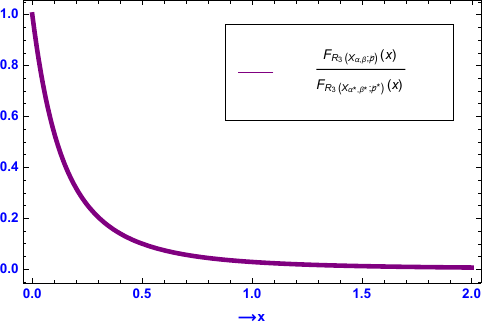}}
			\subfigure[]{\label{figure6(b)}\includegraphics[width=3.2in]{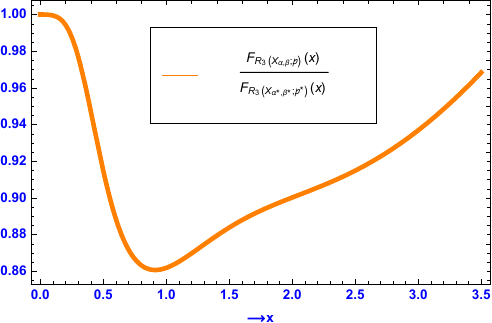}}
			\caption{(a) Plot of the ratio of the cdfs of $R_3(\boldsymbol{X}_{\boldsymbol{\alpha},\boldsymbol{\beta}};\boldsymbol{p})$ and $R_3(\boldsymbol{X}_{\boldsymbol{\alpha^*},\boldsymbol{\beta^*}};\boldsymbol{p^*})$ in Example \ref{example3.7}.
				(b) Plot of the ratio of the cdfs of $R_3(\boldsymbol{X}_{\boldsymbol{\alpha},\boldsymbol{\beta}};\boldsymbol{p})$ and $R_3(\boldsymbol{X}_{\boldsymbol{\alpha^*},\boldsymbol{\beta^*}};\boldsymbol{p^*})$ in Counterexample \ref{counterexample3.6}.}
		\end{center}
	\end{figure}
	
	In the next theorem, under some certain parameter restrictions, likelihood ratio ordering between two MRVs $R_n(\boldsymbol{X}_{\boldsymbol{\alpha},\boldsymbol{\beta}};\boldsymbol{p})$ and $R_n(\boldsymbol{X}_{\boldsymbol{\alpha^*},\boldsymbol{\beta^*}};\boldsymbol{p^*})$ is presented.
	
	\begin{theorem}\label{theorem3.12}
		Consider $f_{R_n(\boldsymbol{X}_{\boldsymbol{\alpha},\boldsymbol{\beta}};\boldsymbol{p})}(x)=\sum_{i=1}^{n}p_i\alpha_i\beta_i(1+x)^{-(\alpha_i+1)}(1-(1+x)^{-\alpha_i})^{\beta_i-1}$ and $f_{R_n(\boldsymbol{X}_{\boldsymbol{\alpha^*},\boldsymbol{\beta^*}};\boldsymbol{p^*})}(x)$ $=\sum_{i=1}^{n}p_i^*\alpha_i^*\beta_i^*(1+x)^{-(\alpha_i^*+1)}(1-(1+x)^{-\alpha_i^*})^{\beta_i^*-1}$ be the pdfs of the MRVs $R_n(\boldsymbol{X}_{\boldsymbol{\alpha},\boldsymbol{\beta}};\boldsymbol{p})$ and $R_n(\boldsymbol{X}_{\boldsymbol{\alpha^*},\boldsymbol{\beta^*}};\boldsymbol{p^*})$, respectively. Then,
		\begin{eqnarray*}
			R_n(\boldsymbol{X}_{\boldsymbol{\alpha},\boldsymbol{\beta}};\boldsymbol{p})\geq_{lr}R_n(\boldsymbol{X}_{\boldsymbol{\alpha^*},\boldsymbol{\beta^*}};\boldsymbol{p^*}),
		\end{eqnarray*}
		provided $\max\{\alpha_1,\ldots,\alpha_n\}\leq\min\{\alpha_1^*,\ldots,\alpha_n^*\}$ and $\min\{\alpha_1\beta_1,\ldots,\alpha_n\beta_n\}\geq\max\{\alpha_1^*\beta_1^*,\ldots,\alpha_n^*\beta_n^*\}$.
	\end{theorem}
	
	\begin{proof}
		See Appendix.
	\end{proof}
	
	The following example illustrates Theorem \ref{theorem3.12}, for $n=3$.
	
	\begin{example}\label{example3.8}
		Set $\boldsymbol{p}=(0.2,0.4,0.4)$, $\boldsymbol{p^*}=(0.3,0.5,0.2)$,    $\boldsymbol{\alpha}=(2,4,6)$, $\boldsymbol{\beta}=(25,13,9)$, $\boldsymbol{\alpha^*}=(8,10,12)$, and $\boldsymbol{\beta^*}=(3,4,1)$. Observe that $\max\{2,4,6\}\leq\min\{8,10,12\}$ and $\max\{24,40,12\}\leq\min\{50,52,54\}$. Thus, all the conditions of Theorem \ref{theorem3.12} are satisfied. Based on the numerical values of the parameters, the ratio of the pdfs of the MRVs $R_3(\boldsymbol{X}_{\boldsymbol{\alpha},\boldsymbol{\beta}};\boldsymbol{p})$ and $R_3(\boldsymbol{X}_{\boldsymbol{\alpha^*},\boldsymbol{\beta^*}};\boldsymbol{p^*})$ is provided in Figure \ref{figure7(a)},
		which readily establishes that $R_3(\boldsymbol{X}_{\boldsymbol{\alpha},\boldsymbol{\beta}};\boldsymbol{p})\geq_{lr}R_3(\boldsymbol{X}_{\boldsymbol{\alpha^*},\boldsymbol{\beta^*}};\boldsymbol{p^*})$, validating the result in Theorem \ref{theorem3.12}.
	\end{example}
	
	Next, present a counterexample to show that the condition $\max\{\alpha_1,\ldots,\alpha_n\}\leq\min\{\alpha_1^*,\ldots,\alpha_n^*\}$ is necessary for establishing the likelihood ratio order in Theorem \ref{theorem3.12}.
	
	\begin{counterexample}\label{counterexample3.7}
		Let $\boldsymbol{p}=(0.1,0.2,0.7)$, $\boldsymbol{p^*}=(0.6,0.3,0.1)$,  $\boldsymbol{\alpha}=(3,6,9)$, $\boldsymbol{\beta}=(14,15,11)$, $\boldsymbol{\alpha^*}=(5,8,12)$, and $\boldsymbol{\beta^*}=(4,2,3)$. Clearly, $\max\{3,6,9\}\nleq\min\{5,8,12\}$ and $\max\{20,16,36\}\leq\min\{42,90,99\}$. Thus, all the conditions of Theorem \ref{theorem3.12} are satisfied except $\max\{3,6,9\}\leq\min\{5,8,12\}$. Now, the ratio of the pdfs of the MRVs
		$R_3(\boldsymbol{X}_{\boldsymbol{\alpha},\boldsymbol{\beta}};\boldsymbol{p})$ and $R_3(\boldsymbol{X}_{\boldsymbol{\alpha^*},\boldsymbol{\beta^*}};\boldsymbol{p^*})$ is depicted in Figure \ref{figure7(b)}, and we see that the ratio is a nonmonotone function in $x>0$. Thus, $R_3(\boldsymbol{X}_{\boldsymbol{\alpha},\boldsymbol{\beta}};\boldsymbol{p})\ngeq_{lr}R_3(\boldsymbol{X}_{\boldsymbol{\alpha^*},\boldsymbol{\beta^*}};\boldsymbol{p^*})$. In other
		words, the likelihood ratio order between $R_3(\boldsymbol{X}_{\boldsymbol{\alpha},\boldsymbol{\beta}};\boldsymbol{p})$ and $R_3(\boldsymbol{X}_{\boldsymbol{\alpha^*},\boldsymbol{\beta^*}};\boldsymbol{p^*})$ in Theorem \ref{theorem3.12} does not hold.
	\end{counterexample}

\begin{figure}
	\begin{center}
		\subfigure[]{\label{figure7(a)}\includegraphics[width=3.2in]{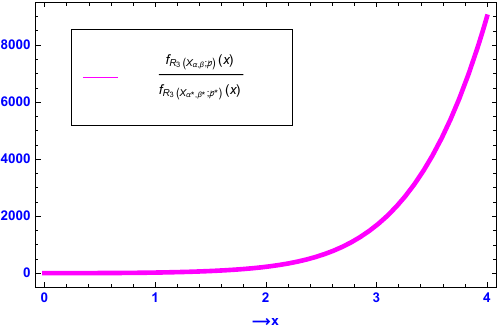}}
		\subfigure[]{\label{figure7(b)}\includegraphics[width=3.08in]{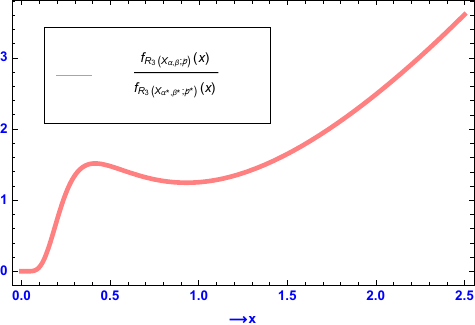}}
		\caption{(a) Plot of the ratio of the pdfs of $R_3(\boldsymbol{X}_{\boldsymbol{\alpha},\boldsymbol{\beta}};\boldsymbol{p})$ and $R_3(\boldsymbol{X}_{\boldsymbol{\alpha^*},\boldsymbol{\beta^*}};\boldsymbol{p^*})$ in Example \ref{example3.8}.
			(b) Plot of the ratio of the pdfs of $R_3(\boldsymbol{X}_{\boldsymbol{\alpha},\boldsymbol{\beta}};\boldsymbol{p})$ and $R_3(\boldsymbol{X}_{\boldsymbol{\alpha^*},\boldsymbol{\beta^*}};\boldsymbol{p^*})$ in Counterexample \ref{counterexample3.7}.}
	\end{center}
\end{figure}

	\section{Concluding remarks}
	In this paper, MMs are considered as suitable tools for analyzing population heterogeneity. We are interested in heterogeneous populations with distinct components such as lifetime. We have derived some sufficient conditions for the
	comparison between two FMMs of $\mathcal{IK}$ distributed components with respect to usual stochastic order, reversed hazard rate order, likelihood ratio order, ageing faster order in terms of reversed hazard rate order corresponding to the heterogeneity in the model
	parameters in the sense of some majorization orders, namely, weakly
	supermajorized and weakly submajorized order.
	Here, we have considered heterogeneity in one parameter, two parameters, and three parameters, respectively. We have presented some numerical examples
	and counterexamples to illustrate the established results in this
	paper.
	
	The presented results of this paper are mostly theoretical. However, one may find some applications of the established results. Below, we consider an example. 
		
		Assume two engineering systems, with components produced by different companies. It is further reasonable to assume that the components have different reliability characteristics. Each of the components can operate in $n$ operational regimes with corresponding different probabilities $p_i$ and $p_i^*$. Let the lifetimes of the $i$th regime have different distributions, say $\mathcal{IK}(\alpha_i,\beta)$ and $\mathcal{IK}(\alpha_i^*,\beta)$, respectively. Then, the important question is: which of these two systems perform better in some stochastic sense? By using Theorem \ref{theorem3.1}, we conclude that under the condition $\boldsymbol{p^*}\preccurlyeq_w \boldsymbol{p}$, the first system performs better than the second system. Similar applications can be found for the other established results.
	
	It is naturally be of interest that one can extend this work
	with respect to some stronger stochastic orders such as hazard rate
	order, or
	with respect to the variability orders, like star order,
	dispersive order, Lorenz order, right-spread order, convex transform order, increasing convex order etc. Future research on the generalizations of these
	findings may be considered.

	\section*{Acknowledgements}
	The financial support (vide D.O.No. F. 14-34/2011 (CPP-II) dated 11.01.2013, F.No. 16-9(June 2019)/2019(NET/CSIR), UGC-Ref.No.:1238/(CSIR-UGC NET JUNE 2019)) from the University Grants Commission (UGC), Government of India, is sincerely acknowledged with thanks by Raju Bhakta.
	
	\section*{Disclosure statement}
	The authors have no conflict of interest/competing interests to
	declare.
	
	\bibliography{ref}
	\appendix
	{\bf Proof of Theorem \ref{theorem3.1}.}\setcounter{equation}{0}
	\begin{proof}
		Denote $\xi(\boldsymbol{p})=\bar{F}_{R_n(\boldsymbol{X}_{\boldsymbol{\alpha},\beta};\boldsymbol{p})}(x)$. Differentiating $\xi(\boldsymbol{p})$ partially with respect to $p_i$, we obtain
		\begin{eqnarray}\label{eq-A1}
			\frac{\partial\xi(\boldsymbol{p})}{\partial p_i}=1-(1-(1+x)^{-\alpha_i})^\beta\geq 0.
		\end{eqnarray}
		Under the assumptions made, we have $(\alpha_i-\alpha_j)(p_i-p_j)\leq0$, for $1\leq i\leq j\leq n$. This implies two possibilities: $\{p_i\geq p_j~\&~\alpha_i\leq\alpha_j\}$ and $\{p_i\leq p_j~\&~\alpha_i\geq\alpha_j\}$. Here, the proof for $\{p_i\geq p_j~\&~\alpha_i\leq\alpha_j\}$ is provided, since the proof under other can follows similarly. Using (\ref{eq-A1}),
		\begin{eqnarray}\label{eq-A2}
			\frac{\partial\xi(\boldsymbol{p})}{\partial p_i}-\frac{\partial \xi(\boldsymbol{p})}{\partial p_j}=\{1-(1-(1+x)^{-\alpha_i})^\beta\}- \{1-(1-(1+x)^{-\alpha_j})^\beta\},
		\end{eqnarray}
		which is clearly non-negative using Lemma \ref{lemma3.1}. Thus, from (\ref{eq-A1}) $\&$ (\ref{eq-A2}), we have
		\begin{eqnarray}\label{eq-A3}
			\frac{\partial\xi(\boldsymbol{p})}{\partial p_i}\geq\frac{\partial\xi(\boldsymbol{p})}{\partial p_j}\geq 0.
		\end{eqnarray}
		Now, the proof is completed using Theorem $3.A.7$ of \cite{marshall2011inequalities}. This completes the proof of the theorem.
	\end{proof}
	
	\noindent{\bf Proof of Theorem \ref{theorem3.2}.}
	\begin{proof}
		Without loss of generality, we consider $0<\beta_1\leq\ldots\leq\beta_n$. Since $(\boldsymbol{\beta},\boldsymbol{p}), (\boldsymbol{\beta},\boldsymbol{p^*})\in\mathcal{L}_n$, thus
		$p_1\geq\ldots\geq p_n>0$ and $p_1^*\geq\ldots\geq p_n^*>0$.
		Now, define
		\begin{eqnarray}\label{eq-A4}
			\phi(\boldsymbol{p})=-\bar{F}_{R_n(\boldsymbol{X}_{\alpha,\boldsymbol{\beta}};\boldsymbol{p})}(x).
		\end{eqnarray}
		Furthermore, for $1\leq i\leq j\leq n$, clearly $\beta_i\leq\beta_j$ and $p_i\geq p_j$. Now, after differentiating (\ref{eq-A4}) partially with respect to $p_i$ and $p_j$, we obtain
		\begin{eqnarray}\label{eq-A5}
			\frac{\partial\phi(\boldsymbol{p})}{\partial p_i}-\frac{\partial \phi(\boldsymbol{p})}{\partial p_j}=\{1-(1-(1+x)^{-\alpha})^{\beta_j}\}-\{1-(1-(1+x)^{-\alpha})^{\beta_i}\}.
		\end{eqnarray}
		Using Lemma \ref{lemma3.1} in (\ref{eq-A5}), it is easy to get that $\frac{\partial\phi(\boldsymbol{p})}{\partial p_i}-\frac{\partial \phi(\boldsymbol{p})}{\partial p_j}$ is non-negative. Thus, for $1\leq i\leq j\leq n$, we obtain
		\begin{eqnarray}\label{eq-A6}
			0\geq\frac{\partial\phi(\boldsymbol{p})}{\partial
				p_i}\geq\frac{\partial\phi(\boldsymbol{p})}{\partial p_j},
		\end{eqnarray}
		since $\frac{\partial\phi(\boldsymbol{p})}{\partial p_i}\leq 0$. Now, using Theorem $3.A.7$ of \cite{marshall2011inequalities}, the proof can be completed.
	\end{proof}

	\noindent{\bf Proof of Theorem \ref{theorem3.3}.}
	\begin{proof}
		Without loss of generality, we assume $p_1\geq\ldots\geq
		p_n>0$. Thus, from the assumptions made, we have $\alpha_1\leq\ldots\leq\alpha_n$ and $\alpha_1^*\leq\ldots\leq\alpha_n^*$. Define
		\begin{eqnarray}\label{eq-A7}
			\psi(\boldsymbol{\alpha})=\bar{F}_{R_n(\boldsymbol{X}_{\boldsymbol{\alpha},\beta};\boldsymbol{p})}(x).
		\end{eqnarray}
		From (\ref{eq-A7}), we obtain the partial derivative of $\psi(\boldsymbol{\alpha})$ with respect to $\alpha_i$ as
		\begin{eqnarray}\label{eq-A8}
			\frac{\partial\psi(\boldsymbol{\alpha})}{\partial\alpha_i}=-\beta\log(1+x)p_i(1+x)^{-\alpha_i}(1-(1+x)^{-\alpha_i})^{\beta-1},
		\end{eqnarray}
		which is clearly non-positive. Thus, $\psi(\boldsymbol{\alpha})$ is decreasing with respect to $\alpha_i$. Further, using Lemma \ref{lemma3.1}, for $1\leq i\leq j\leq n$, we obtain after some calculations as
		\begin{eqnarray}\label{eq-A9}
			\frac{\partial\psi(\boldsymbol{\alpha})}{\partial\alpha_i}\leq\frac{\partial\psi(\boldsymbol{\alpha})}{\partial\alpha_j}\leq 0.
		\end{eqnarray}
		Now, from Lemma $2.4$ of \cite{bhakta2023stochastic} it can be shown that $\psi(\boldsymbol{\alpha})$ is Schur-convex with respect to $\boldsymbol{\alpha}$. Thus, the remaining proof of the theorem follows from Theorem $3.A.8$ of \cite{marshall2011inequalities}, p. $87$. This completes the proof of the Theorem.
	\end{proof}

	\noindent{\bf Proof of Theorem \ref{theorem3.4}.}
	\begin{proof}
		The theorem will be proved if the conditions of Lemma \ref{lemma2.1} are satisfied. Here, $\bar{F}_{R_2(\boldsymbol{X}_{\boldsymbol{\alpha},\beta};\boldsymbol{p})}(x)$
		is clearly permutation invariant on $\mathcal{L}_2$. Further,
		\begin{eqnarray}\label{eq-A10}
			\Delta_1(\boldsymbol{p},\boldsymbol{\alpha})&=&
			(p_1-p_2)\left(\frac{\partial\bar{F}_{R_2(\boldsymbol{X}_{\boldsymbol{\alpha},\beta};\boldsymbol{p})}(x)}{\partial p_1}-\frac{\partial\bar{F}_{R_2(\boldsymbol{X}_{\boldsymbol{\alpha},\beta};\boldsymbol{p})}(x)}{\partial p_2}\right)\nonumber\\
			&&+(\alpha_1-\alpha_2)\left(\frac{\partial\bar{F}_{R_2(\boldsymbol{X}_{\boldsymbol{\alpha},\beta};\boldsymbol{p})}(x)}{\partial\alpha_1}-\frac{\partial\bar{F}_{R_2(\boldsymbol{X}_{\boldsymbol{\alpha},\beta};\boldsymbol{p})}(x)}{\partial\alpha_2}\right)\nonumber\\
			&=&(p_1-p_2)\bigg([1-(1-(1+x)^{-\alpha_1})^\beta]-[1-(1-(1+x)^{-\alpha_2})^\beta]\bigg)\nonumber\\
			&&+\beta\log(1+x)(\alpha_1-\alpha_2)\bigg(p_2(1+x)^{-\alpha_2}(1-(1+x)^{-\alpha_2})^{\beta-1})\nonumber\\
			&&-p_1(1+x)^{-\alpha_1}(1-(1+x)^{-\alpha_1})^{\beta-1}\bigg).
		\end{eqnarray}
		Under the assumption made, we have $(\boldsymbol{p},\boldsymbol{\alpha})\in\mathcal{L}_2$, that is, either $\{p_1\geq p_2~\&~\alpha_1\leq\alpha_2\}$ holds or $\{p_1\leq p_2~\&~\alpha_1\geq\alpha_2\}$ holds. We present the proof for the case $\{p_1\geq p_2~\&~\alpha_1\leq\alpha_2\}$. The proof for the other case $\{p_1\leq p_2~\&~\alpha_1\geq\alpha_2\}$ is similar. Using Lemma \ref{lemma3.1}, the first term in the right hand side of (\ref{eq-A10}) is clearly non-negative, since $\alpha_1\leq\alpha_2$. Further, using Lemma \ref{lemma3.2}, the second term in (\ref{eq-A10}) can be shown to be non-negative. Thus,
		\begin{eqnarray}\label{eq-A11}
			\Delta_1(\boldsymbol{p},\boldsymbol{\alpha})\geq 0,
		\end{eqnarray}
		satisfying the second condition of Lemma \ref{lemma2.1}. This completes the proof of the Theorem.
	\end{proof}

	\noindent{\bf Proof of Theorem \ref{theorem3.7}.}
	\begin{proof}
		We note that the proof of this theorem follows using Lemma \ref{lemma2.1}. It is easy to notice that $R_2(\boldsymbol{X}_{\alpha,\boldsymbol{\beta}};\boldsymbol{p})$ is permutation invariant on $\mathcal{L}_2$, confirming the condition $(i)$ of Lemma \ref{lemma2.1}. To check condition $(ii)$ of Lemma \ref{lemma2.1}, we consider
		\begin{eqnarray}\label{eq-A12}
			\Delta_2(\boldsymbol{p},\boldsymbol{\beta})&=&
			(p_1-p_2)\left(\frac{\partial\bar{F}_{R_2(\boldsymbol{X}_{\alpha,\boldsymbol{\beta}};\boldsymbol{p})}(x)}{\partial p_1}-\frac{\partial\bar{F}_{R_2(\boldsymbol{X}_{\alpha,\boldsymbol{\beta}};\boldsymbol{p})}(x)}{\partial p_2}\right)\nonumber\\
			&&+(\beta_1-\beta_2)\left(\frac{\partial\bar{F}_{R_2(\boldsymbol{X}_{\alpha,\boldsymbol{\beta}};\boldsymbol{p})}(x)}{\partial\beta_1}-\frac{\partial\bar{F}_{R_2(\boldsymbol{X}_{\alpha,\boldsymbol{\beta}};\boldsymbol{p})}(x)}{\partial\beta_2}\right)\nonumber\\
			&=&(p_1-p_2)\Big([1-(1-(1+x)^{-\alpha})^{\beta_1}]-[1-(1-(1+x)^{-\alpha})^{\beta_2}]\Big)\nonumber\\
			&&+(\beta_1-\beta_2)\log(1-(1+x)^{-\alpha})\Big(p_2(1-(1+x)^{-\alpha})^{\beta_2}-p_1(1-(1+x)^{-\alpha})^{\beta_1}\Big).
		\end{eqnarray}
		Here, we have assumed that $(\boldsymbol{p},\boldsymbol{\beta})\in\mathcal{L}_2$, implying either $\{p_1\geq p_2~\&~\beta_1\leq\beta_2\}$ or $\{p_1\leq p_2~\&~\beta_1\geq\beta_2\}$. The proof will be presented for $\{p_1\geq p_2~\&~\beta_1\leq\beta_2\}$. For other case, it is similar. For $p_1\geq p_2$ and $\beta_1\leq\beta_2$, using Lemma \ref{lemma3.1}, it can be shown after some calculations that the first term of right hand side in (\ref{eq-A12}) is non-positive. Further, from Lemma \ref{lemma3.3}, the second term in the right hand side of (\ref{eq-A12}) is non-positive when $p_1\geq p_2$ and $\beta_1\leq\beta_2$. Combining these, we obtain
		\begin{eqnarray}\label{eq-A13}
			\Delta_2(\boldsymbol{p},\boldsymbol{\beta})\leq 0,
		\end{eqnarray}
		which proves that the second condition of Lemma \ref{lemma2.1} is also satisfied. Thus, the proof is completed.
	\end{proof}

	\noindent{\bf Proof of Theorem \ref{theorem3.10}.}
	\begin{proof}
		Here, it is sufficient to establish that the ratio $\frac{\varphi_1(x)}{\varphi_2(x)}=\frac{\tilde{r}_{R_n(\boldsymbol{X}_{\alpha,\boldsymbol{\beta}};\boldsymbol{p})}(x)}{\tilde{r}_{R_n(\boldsymbol{X}_{\alpha,\boldsymbol{\beta^*}};\boldsymbol{p^*})}(x)}$ is non-increasing with respect to $x$, where
		\begin{eqnarray}\label{eq-A14}
			\frac{\varphi_1(x)}{\varphi_2(x)}=\frac{f_{R_n(\boldsymbol{X}_{\alpha,\boldsymbol{\beta}};\boldsymbol{p})}(x)F_{R_n(\boldsymbol{X}_{\alpha,\boldsymbol{\beta^*}};\boldsymbol{p^*})}(x)}{f_{R_n(\boldsymbol{X}_{\alpha,\boldsymbol{\beta^*}};\boldsymbol{p^*})}(x)F_{R_n(\boldsymbol{X}_{\alpha,\boldsymbol{\beta}};\boldsymbol{p})}(x)}.
		\end{eqnarray}
		On differentiation (\ref{eq-A14}) with respect to $x$, we obtain
		\begin{eqnarray}\label{eq-A15}
			\frac{\partial}{\partial x}\left\lbrace \frac{\tilde{r}_{R_n(\boldsymbol{X}_{\alpha,\boldsymbol{\beta}};\boldsymbol{p})}(x)}{\tilde{r}_{R_n(\boldsymbol{X}_{\alpha,\boldsymbol{\beta^*}};\boldsymbol{p^*})}(x)} \right\rbrace&\stackrel{sign}{=}&\varphi_1^{\prime}(x)\varphi_2(x)-\varphi_1(x)\varphi_2^{\prime}(x)\nonumber\\
			&=&f_{R_n(\boldsymbol{X}_{\alpha,\boldsymbol{\beta}};\boldsymbol{p})}^{\prime}(x)F_{R_n(\boldsymbol{X}_{\alpha,\boldsymbol{\beta}};\boldsymbol{p})}(x)f_{R_n(\boldsymbol{X}_{\alpha,\boldsymbol{\beta^*}};\boldsymbol{p^*})}(x)F_{R_n(\boldsymbol{X}_{\alpha,\boldsymbol{\beta^*}};\boldsymbol{p^*})}(x)\nonumber\\
			&&+f_{R_n(\boldsymbol{X}_{\alpha,\boldsymbol{\beta}};\boldsymbol{p})}(x)f_{R_n(\boldsymbol{X}_{\alpha,\boldsymbol{\beta^*}};\boldsymbol{p^*})}^2(x)F_{R_n(\boldsymbol{X}_{\alpha,\boldsymbol{\beta}};\boldsymbol{p})}(x)\nonumber\\
			&&-f_{R_n(\boldsymbol{X}_{\alpha,\boldsymbol{\beta}};\boldsymbol{p})}^2(x)f_{R_n(\boldsymbol{X}_{\alpha,\boldsymbol{\beta^*}};\boldsymbol{p^*})}(x)F_{R_n(\boldsymbol{X}_{\alpha,\boldsymbol{\beta^*}};\boldsymbol{p^*})}(x)\nonumber\\
			&&-f_{R_n(\boldsymbol{X}_{\alpha,\boldsymbol{\beta}};\boldsymbol{p})}(x)F_{R_n(\boldsymbol{X}_{\alpha,\boldsymbol{\beta}};\boldsymbol{p})}(x)f_{R_n(\boldsymbol{X}_{\alpha,\boldsymbol{\beta^*}};\boldsymbol{p^*})}^{\prime}(x)F_{R_n(\boldsymbol{X}_{\alpha,\boldsymbol{\beta^*}};\boldsymbol{p^*})}(x)\nonumber\\
			&=&\xi(x),~\text{(say)}.
		\end{eqnarray}
		Our goal is to show that $\xi(x)$ is non-positive. Now, using the model assumptions, from (\ref{eq-A15}), we obtain
		\allowdisplaybreaks{\begin{eqnarray}
				\xi(x)&=&\Biggl\{\Bigg( \sum_{i=1}^{n}p_i\alpha\beta_i\big[\alpha(\beta_i-1)(1+x)^{-2(\alpha+1)}(1-(1+x)^{-\alpha})^{\beta_i-2}\nonumber\\
				&&-(\alpha+1)(1+x)^{-(\alpha+2)}(1-(1+x)^{-\alpha})^{\beta_i-1}\big]\Bigg)\times\Bigg(\sum_{i=1}^{n}p_i(1-(1+x)^{-\alpha})^{\beta_i}\Bigg)\nonumber\\
				&&\times\Bigg(\sum_{i=1}^{n}p_i^*\alpha\beta_i^*(1+x)^{-(\alpha+1)}(1-(1+x)^{-\alpha})^{\beta_i^*-1}\Bigg)\times\Bigg(\sum_{i=1}^{n}p_i^*(1-(1+x)^{-\alpha})^{\beta_i^*}\Bigg)\Biggr\}\nonumber\\
				&&+\Biggl\{\Bigg(\sum_{i=1}^{n}p_i\alpha\beta_i(1+x)^{-(\alpha+1)}(1-(1+x)^{-\alpha})^{\beta_i-1}\Bigg)\times\Bigg(\sum_{i=1}^{n}p_i(1-(1+x)^{-\alpha})^{\beta_i}\Bigg)\nonumber\\
				&&\times\Bigg(\sum_{i=1}^{n}p_i^*\alpha\beta_i^*(1+x)^{-(\alpha+1)}(1-(1+x)^{-\alpha})^{\beta_i^*-1}\Bigg)^2\Biggr\}-\Biggl\{\Bigg(\sum_{i=1}^{n}p_i\alpha\beta_i(1+x)^{-(\alpha+1)}\nonumber\\
				&&(1-(1+x)^{-\alpha})^{\beta_i-1}\Bigg)^2\times\Bigg(\sum_{i=1}^{n}p_i^*\alpha\beta_i^*(1+x)^{-(\alpha+1)}(1-(1+x)^{-\alpha})^{\beta_i^*-1}\Bigg)\nonumber\\
				&&\times\Bigg(\sum_{i=1}^{n}p_i^*(1-(1+x)^{-\alpha})^{\beta_i^*}\Bigg)\Biggr\}-\Biggl\{\Bigg(\sum_{i=1}^{n}p_i\alpha\beta_i(1+x)^{-(\alpha+1)}(1-(1+x)^{-\alpha})^{\beta_i-1}\Bigg)\nonumber\\
				&&\times\Bigg(\sum_{i=1}^{n}p_i(1-(1+x)^{-\alpha})^{\beta_i}\Bigg)\times\Bigg(\sum_{i=1}^{n}p_i^*\alpha\beta_i^*\Big[\alpha(\beta_i^*-1)(1+x)^{-2(\alpha+1)}(1-(1+x)^{-\alpha})^{\beta_i^*-2}\nonumber\\
				&&-(\alpha+1)(1+x)^{-(\alpha+2)}(1-(1+x)^{-\alpha})^{\beta_i^*-1}\Big]\Bigg)\times\Bigg(\sum_{i=1}^{n}p_i^*(1-(1+x)^{-\alpha})^{\beta_i^*}\Bigg)\Biggr\}\nonumber\\
				&=&\sum_{i=1}^{n}\sum_{j=1}^{n}\sum_{k=1}^{n}\sum_{l=1}^{n}p_ip_jp_k^*p_l^*\alpha^2\beta_i\beta_k^*(1+x)^{-(2\alpha+3)}(1-(1+x)^{-\alpha})^{\beta_i-1}\Bigl\{\alpha(\beta_i-1)(1+x)^{-\alpha}\nonumber\\
				&&(1-(1+x)^{-\alpha})^{-1}-(\alpha+1)\Bigr\}(1-(1+x)^{-\alpha})^{\beta_j}(1-(1+x)^{-\alpha})^{\beta_k^*+\beta_l^*-1}\nonumber\\
				&&+\sum_{i=1}^{n}\sum_{j=1}^{n}\sum_{k=1}^{n}\sum_{l=1}^{n}p_ip_jp_k^*p_l^*\alpha^3\beta_i\beta_k^*\beta_l^*(1+x)^{-3(\alpha+1)}(1-(1+x)^{-\alpha})^{\beta_i+\beta_j+\beta_k^*+\beta_l^*-3}\nonumber\\
				&&-\sum_{i=1}^{n}\sum_{j=1}^{n}\sum_{k=1}^{n}\sum_{l=1}^{n}p_ip_jp_k^*p_l^*\alpha^3\beta_i\beta_j\beta_k^*(1+x)^{-3(\alpha+1)}(1-(1+x)^{-\alpha})^{\beta_i+\beta_j+\beta_k^*+\beta_l^*-3}\nonumber\\
				&&-\sum_{i=1}^{n}\sum_{j=1}^{n}\sum_{k=1}^{n}\sum_{l=1}^{n}p_ip_jp_k^*p_l^*\alpha^2\beta_i\beta_k^*(1+x)^{-(2\alpha+3)}(1-(1+x)^{-\alpha})^{\beta_i+\beta_j+\beta_k^*+\beta_l^*-2}\nonumber\\
				&&\Bigl\{\alpha(\beta_k^*-1)(1+x)^{-\alpha}(1-(1+x)^{-\alpha})^{-1}-(\alpha+1)\Bigr\}\nonumber
		\end{eqnarray}}
		\begin{eqnarray}
			&=&\sum_{i=1}^{n}\sum_{j=1}^{n}\sum_{k=1}^{n}\sum_{l=1}^{n}p_ip_jp_k^*p_l^*\alpha^2\beta_i\beta_k^*(1+x)^{-(2\alpha+3)}(1-(1+x)^{-\alpha})^{\beta_i+\beta_j+\beta_k^*+\beta_l^*-2}\Delta_{i,j,k,l}(x),\nonumber
		\end{eqnarray}
		where
		\begin{eqnarray}\label{eq-A16}
			\Delta_{i,j,k,l}(x)&=&\Bigl\{\alpha(\beta_i-1)(1+x)^{-\alpha}(1-(1+x)^{-\alpha})^{-1}-(\alpha+1)\Bigr\}\nonumber\\
			&&-\Bigl\{\alpha(\beta_k^*-1)(1+x)^{-\alpha}(1-(1+x)^{-\alpha})^{-1}-(\alpha+1)\Bigr\}\nonumber\\
			&&+\alpha\beta_l^*(1+x)^{-\alpha}(1-(1+x)^{-\alpha})^{-1}-\alpha\beta_j(1+x)^{-\alpha}(1-(1+x)^{-\alpha})^{-1}.
		\end{eqnarray}
		Consider $1\leq i,j\leq n$. Under the assumptions made, from (\ref{eq-A16}), we obtain $\Delta_{i,j,k,l}(x)\leq 0$, and hence $\xi(x)\leq0$. This completes the proof of the theorem.
	\end{proof}
	\noindent{\bf Proof of Theorem \ref{theorem3.11}.}
	\begin{proof}
		The proof will be completed if we show that
		\begin{eqnarray}\label{eq-A17}
			\frac{F_{R_n(\boldsymbol{X}_{\boldsymbol{\alpha},\boldsymbol{\beta}};\boldsymbol{p})}(x)}{F_{R_n(\boldsymbol{X}_{\boldsymbol{\alpha^*},\boldsymbol{\beta^*}};\boldsymbol{p^*})}(x)}=\frac{\sum_{i=1}^{n}p_i(1-(1+x)^{-\alpha_i})^{\beta_i}}{\sum_{i=1}^{n}p_i^*(1-(1+x)^{-\alpha_i^*})^{\beta_i^*}}=\frac{\phi_1(x)}{\phi_2(x)},~\text{(say)},
		\end{eqnarray}
		is non-increasing with respect to $x>0$. The partial derivative of (\ref{eq-A17}) with respect to $x$ is
		\begin{eqnarray}\label{eq-A18}
			\frac{\partial}{\partial x}\left\lbrace \frac{F_{R_n(\boldsymbol{X}_{\boldsymbol{\alpha},\boldsymbol{\beta}};\boldsymbol{p})}(x)}{F_{R_n(\boldsymbol{X}_{\boldsymbol{\alpha^*},\boldsymbol{\beta^*}};\boldsymbol{p^*})}(x)}\right\rbrace&\stackrel{sign}{=}&\phi^{\prime}_1(x)\phi_2(x)-\phi_1(x)\phi^{\prime}_2(x)\nonumber\\
			&=&\xi(x),
		\end{eqnarray}
		where
		\begin{eqnarray}\label{eq-A19}
			\xi(x)&=&\left(\sum_{i=1}^{n}p_i\alpha_i\beta_i(1+x)^{-(\alpha_i+1)}(1-(1+x)^{-\alpha_i})^{\beta_i-1}\right) \left(\sum_{i=1}^{n}p_i^*(1-(1+x)^{-\alpha_i^*})^{\beta_i^*}\right)\nonumber\\
			&&-\left(\sum_{i=1}^{n}p_i(1-(1+x)^{-\alpha_i})^{\beta_i} \right)\left(\sum_{i=1}^{n}p_i^*\alpha_i^*\beta_i^*(1+x)^{-(\alpha_i^*+1)}(1-(1+x)^{-\alpha_i^*})^{\beta_i^*-1}\right)\nonumber\\
			&=&\sum_{i=1}^{n}\sum_{j=1}^{n}p_ip_j^*(1-(1+x)^{-\alpha_i})^{\beta_i-1}(1-(1+x)^{-\alpha_j^*})^{\beta_j^*-1}\big[\alpha_i\beta_i(1+x)^{-(\alpha_i+1)}\nonumber\\
			&&(1-(1+x)^{-\alpha_j^*})-\alpha_j^*\beta_j^*(1+x)^{-(\alpha_j^*+1)}(1-(1+x)^{-\alpha_i})\big].
		\end{eqnarray}
		For $1\leq i,j\leq n$, under the assumptions made, it can be shown that $\xi(x)\leq 0$, as desired.
	\end{proof}
	\noindent{\bf Proof of Theorem \ref{theorem3.12}.}
	\begin{proof}
		To prove the result, it is required to establish that
		\begin{eqnarray}\label{eq-A20}
			\frac{f_{R_n(\boldsymbol{X}_{\boldsymbol{\alpha},\boldsymbol{\beta}};\boldsymbol{p})}(x)}{f_{R_n(\boldsymbol{X}_{\boldsymbol{\alpha^*},\boldsymbol{\beta^*}};\boldsymbol{p^*})}(x)}=\frac{\sum_{i=1}^{n}p_i\alpha_i\beta_i(1+x)^{-(\alpha_i+1)}(1-(1+x)^{-\alpha_i})^{\beta_i-1}}{\sum_{i=1}^{n}p_i^*\alpha_i^*\beta_i^*(1+x)^{-(\alpha_i^*+1)}(1-(1+x)^{-\alpha_i^*})^{\beta_i^*-1}}=\xi(x),~\text{(say)},
		\end{eqnarray}
		is non-decreasing with respect to $x>0$. The derivative of $\xi(x)$
		with respect to $x$ is obtained as
		\allowdisplaybreaks{\begin{eqnarray}\label{eq-A21}
				\xi^{\prime}(x)&\stackrel{sign}{=}&\Bigg(\sum_{i=1}^{n}p_i\alpha_i\beta_i\big[\alpha_i(\beta_i-1)(1+x)^{-2(\alpha_i+1)}(1-(1+x)^{-\alpha_i})^{\beta_i-2}\nonumber\\
				&&-(\alpha_i+1)(1+x)^{-(\alpha_i+2)}(1-(1+x)^{-\alpha_i})^{\beta_i-1}\big]\Bigg)\nonumber\\
				&&\times\Bigg(\sum_{i=1}^{n}p_i^*\alpha_i^*\beta_i^*(1+x)^{-(\alpha_i^*+1)}(1-(1+x)^{-\alpha_i^*})^{\beta_i^*-1}\Bigg)\nonumber\\
				&&-\Bigg(\sum_{i=1}^{n}p_i\alpha_i\beta_i(1+x)^{-(\alpha_i+1)}(1-(1+x)^{-\alpha_i})^{\beta_i-1}\Bigg)\nonumber\\
				&&\times\Bigg(\sum_{i=1}^{n}p_i^*\alpha_i^*\beta_i^*\big[\alpha_i^*(\beta_i^*-1)(1+x)^{-2(\alpha_i^*+1)}(1-(1+x)^{-\alpha_i^*})^{\beta_i^*-2}\nonumber\\
				&&-(\alpha_i^*+1)(1+x)^{-(\alpha_i^*+2)}(1-(1+x)^{-\alpha_i^*})^{\beta_i^*-1}\big]\Bigg)\nonumber\\
				&=&\Bigg(\sum_{i=1}^{n}\sum_{j=1}^{n}p_ip_j^*\alpha_i\beta_i\alpha_j^*\beta_j^*(1+x)^{-(\alpha_j^*+1)}(1-(1+x)^{-\alpha_j^*})^{\beta_j^*-1}\nonumber\\
				&&\times\big[\alpha_i(\beta_i-1)(1+x)^{-2(\alpha_i+1)}(1-(1+x)^{-\alpha_i})^{\beta_i-2}\nonumber\\
				&&-(\alpha_i+1)(1+x)^{-(\alpha_i+2)}(1-(1+x)^{-\alpha_i})^{\beta_i-1}\big]\Bigg)\nonumber\\
				&&-\Bigg(\sum_{i=1}^{n}\sum_{j=1}^{n}p_ip_j^*\alpha_i\beta_i\alpha_j^*\beta_j^*(1+x)^{-(\alpha_i+1)}(1-(1+x)^{-\alpha_i})^{\beta_i-1}\nonumber\\
				&&\times\big[\alpha_j^*(\beta_j^*-1)(1+x)^{-2(\alpha_j^*+1)}(1-(1+x)^{-\alpha_j^*})^{\beta_j^*-2}\nonumber\\
				&&-(\alpha_j^*+1)(1+x)^{-(\alpha_j^*+2)}(1-(1+x)^{-\alpha_j^*})^{\beta_j^*-1}\big]\Bigg).
		\end{eqnarray}}
		Consider $1\leq i,j\leq n$. Under the assumptions made, from (\ref{eq-A21}), we obtain $\xi^{\prime}(x)\geq 0$, implying that $\xi(x)$ is non-decreasing with respect to $x>0$. Hence the result follows.
	\end{proof}

\end{document}